\documentclass{amsart}
\usepackage{graphicx} 

\usepackage[draft]{say}
\definecolor{egyptianblue}{rgb}{0.06, 0.2, 0.65}

\usepackage[dvipsnames]{xcolor}
\definecolor{MHcol}{RGB}{100, 180, 255}

\usepackage{amsthm}
\usepackage{amssymb}
\usepackage{amsmath}
\usepackage{parskip}
\usepackage{tikz}
\usepackage{tikz-cd}
\usepackage{bbm}
\usepackage{enumitem}
\usepackage{xcolor}
\usepackage{parskip}
\usepackage{import}
\usepackage{yhmath}
\usepackage{aliascnt}
\usepackage{standalone}
\usepackage{caption}
\usepackage{subcaption}
\usepackage{mathtools}

\usepackage{comment}

\usepackage[margin=1in,marginparwidth=0.8in, marginparsep=0.1in]{geometry}

\usepackage[bookmarks=true, bookmarksopen=true,%
bookmarksdepth=3,bookmarksopenlevel=2,%
colorlinks=true,%
linkcolor=blue,%
citecolor=blue,%
filecolor=blue,%
menucolor=blue,%
urlcolor=blue]{hyperref}

\usepackage{mathtools}

\DeclareMathOperator*{\colim}{colim}

\DeclareMathOperator{\Hom}{Hom}

\DeclareMathOperator{\Sh}{Sh}
\DeclareMathOperator{\End}{End}

\newcommand{\Mod}{\mathrm{Mod}}

\DeclareMathOperator{\Perf}{Perf}
\DeclareMathOperator{\Fuk}{Fuk}

\newcommand{\op}{\mathrm{op}}

\newcommand{\cA}{\mathcal{A}}
\newcommand{\cB}{\mathcal{B}}
\newcommand{\cC}{\mathcal{C}}
\newcommand{\cK}{\mathcal{K}}
\newcommand{\cF}{\mathcal{F}}

\newcommand{\cS}{\mathcal{S}}

\newcommand{\bR}{\mathbb{R}}

\newcommand{\bZ}{\mathbb{Z}}

\newcommand{\fg}{\mathfrak{g}}

\DeclareMathOperator{\starp}{star^{+}}
\DeclareMathOperator{\sta}{star}

\DeclareMathOperator{\coh}{coh}

\DeclareMathOperator{\fib}{fib}
\DeclareMathOperator{\cone}{cone}

\let\SS\relax
\DeclareMathOperator{\SS}{SS}

\DeclareMathOperator{\cofib}{cofib}

\DeclareMathOperator{\thick}{thick}
\DeclareMathOperator{\Loc}{Loc}

\DeclareMathOperator{\Ext}{Ext}
\DeclareMathOperator{\Idem}{Idem}
\DeclareMathOperator{\gr}{gr}
\DeclareMathOperator{\qgr}{qgr}

\newcommand{\cX}{\mathcal{X}}
\newcommand{\cR}{\mathcal{R}}
\newcommand{\cO}{\mathcal{O}}

\newtheorem{dummy}{dummy}[section]
\newtheorem{lemma}[dummy]{Lemma}
\newtheorem{theorem}[dummy]{Theorem}

\newtheorem{conjecture}[dummy]{Conjecture}
\newtheorem{wrong}[dummy]{Nonjecture}
\newtheorem{corollary}[dummy]{Corollary}
\newtheorem{proposition}[dummy]{Proposition}

\theoremstyle{definition}
\newtheorem{definition}[dummy]{Definition}

\newtheorem{example}[dummy]{Example}
\newtheorem{remark}[dummy]{Remark}

\newtheorem{question}[dummy]{Question}

\title{A proper dg algebra which does not cogenerate}
\author{Mingyuan Hu} \author{Vivek Shende} \author{Dinglong Wang}

\begin{document}

\maketitle

\begin{abstract}
    Keller's
    strong form of the homological conjectures asserts: any finite dimensional algebra cogenerates its unbounded derived category of modules.  Here we record an example of a (coconnective) dg algebra with finite dimensional cohomology, which does not cogenerate its module category, along with some related phenomena: a smooth dg category whose dualizing bimodule fails to be nondegenerate, and a nontrivial fully faithful left Calabi-Yau morphism. 
    All  examples and most proofs were produced by ChatGPT. 

    In an appendix we explain the  reason  we were looking for such examples: their existence would follow from the existence of Weinstein symplectic manifolds which failed to satisfy Arnol'd's chord conjecture.
\end{abstract}

    
\section{Introduction}

There is a family of much studied `homological' conjectures about the categories of modules over finite dimensional algebras, going back at least to 1960 \cite{Bass}; we refer to \cite{Happel} for an overview. Keller gave a modern strengthening in terms of generation of the unbounded derived category.  

We say that $X$ generates a dg category $\mathcal{C}$ if, for $M \in \mathcal{C}$, we have $\Hom(X, M) = 0 \implies M = 0$.  
Here, as throughout the article, all functors like $\Hom$, $\otimes$, etc. mean their derived variants; similarly, when we refer to categories of modules we mean the dg category of unbounded dg modules.

\begin{conjecture} \label{keller conjecture} \cite{Keller-unbounded-homological, Rickard} Let $A$ be a finite dimensional algebra over a field $k$. 
Then $A^\vee = \Hom_k(A,k)$ generates the category $\Mod_A$ of right $A$-modules.
\end{conjecture}


As observed by Rickard, if $A^{\op}$ is finite dimensional and $(A^{\op})^\vee$ generates, then $A$ cogenerates the category $\Mod_{A}$, i.e. $\Hom_A(M, A) = 0 \implies M = 0$ \cite[Proposition 5.1]{Rickard}.  He also showed that an algebra which cogenerates satisfies the finitistic dimension conjecture~\cite[Proposition 5.2]{Rickard}, which is in turn known to imply the various other homological conjectures, and that the condition that  $A$ cogenerates depends only on $\Mod_A$.  

One might naturally formulate: 

\begin{wrong} \label{proper dgas cogenerate}
    Any dg algebra over a field with finite dimensional cohomology cogenerates. 
\end{wrong}

We give one positive result in this direction: 

\begin{proposition}
    \label{prop: bounded below}
    Let $A$ be a coconnective ($H^{<0}(A) = 0$) dg algebra with $H^0(A) = k$ and $H^*(A)$ bounded above.  Then $\Hom_A(M, A) \ne 0$ for every $M \in \Mod_A$ with $H^*(M) \ne 0$ bounded below.
\end{proposition}

However, Nonjecture \ref{proper dgas cogenerate} is false.  Here we record an example, found by ChatGPT, prompted to consider the ring which was used by Schulz to provide a counterexample to Tachikawa's conjecture \cite{Schulz1994}.

\begin{example} \label{dga which does not cogenerate}
    Let $q \in k^\times$ be not a root of unity, and consider the graded algebra (with zero differential)
    \[
    \Lambda = k\langle x, y \rangle / (x^2, y^2, xy - qyx), \qquad |x| = |y| = 2,
    \]
    the right $\Lambda$-module $V = (\Lambda/(x+y)\Lambda)[-1]$, and the dg algebra 
    \[
    A = \End_\Lambda(\Lambda \oplus V).
    \]
    Then $A$ is coconnective, with $H^0(A) = k \oplus k$ and $\dim_k H^*(A) = 11$.  
    Here we use cohomological grading, and ``coconnective'' means $H^i(A) = 0$ for $i < 0$. 
    Let $e \in A$ be the idempotent projecting onto the summand $\Lambda$, and 
    \[
    M = \cofib\left( Ae \otimes_{eAe} eA \longrightarrow A \right) \in \Mod_A .
    \]
    Then $M \ne 0$ (in fact $H^*(M) = k \oplus k[-1]$), but $\Hom_A(M, A) = 0$.  Thus $A$ does not cogenerate; it does not even cogenerate the pseudo-perfect modules (i.e. modules with finite dimensional cohomology).  Details are in Section \ref{sec: example dga}.  In Appendix \ref{app:finite-dg-model}, we show that in fact $A$ and $M$ admit (not just cohomologically) finite dimensional models. 
\end{example}

The cogeneration question is related to nondegeneracy of smooth-but-not-proper categories, and to the existence of fully faithful left Calabi-Yau structures:

\begin{proposition} \label{cogeneration versus nondegeneracy}
    Let $\mathcal{C}$ be a smooth proper dg category, and 
    $0 \to \mathcal{K} \to \mathcal{C} \to \mathcal{S} \to 0$ a localization sequence (it follows immediately that $\mathcal{K}$ is proper and $\mathcal{S}$ is smooth).  If $\Mod_{\mathcal{K}}$ is cogenerated by compact objects, then $\mathcal{S}^!$ is nondegenerate: if $\cS^!(M, \cdot) = 0$ then $M = 0$. 
\end{proposition}

Note that if $\cK$ is generated by some compact objects $\{k_i\}$, then compact cogeneration of $\cK$ is equivalent to asking that $A = \End(\oplus k_i)$ cogenerates; in case there are finitely many $k_i$, said $A$ is necessarily a dg algebra with finite dimensional cohomology.

\begin{proposition}\label{universal ff lcy}  (Lemmas \ref{lem: kernel of IDB} and \ref{lem: key})
    Let $\mathcal{S}$ be a smooth dg category, and $\cS^!$ be its dualizing bimodule.  Then the image of any fully faithful left Calabi-Yau morphism to $\cS$ generates the kernel of $(-)\otimes_\cS\cS^!$ as a localizing subcategory of $\Mod_\cS$.
    A sort of converse: if $\mathcal{Y}$ is contained in the kernel of $(-)\otimes_\cS \cS^!$, then the relative left Calabi-Yau completion  $\Pi_{n-1}(\mathcal{Y}) \to \Pi_{n}(\mathcal{S}, \mathcal{Y})$ is fully faithful. 
\end{proposition}

In fact, ChatGPT also found the following: 

\begin{example} \label{eg:nontrivial ff lcy}
    Take 
    $
    R = k\langle x, y, \delta \rangle / ( [x, y], [\delta, x] - 1, [\delta, y] + y )
    $
    and set
    \[
    M_\lambda = R/(\delta-\lambda)R, \qquad
    \cB = \Idem(\Perf(R)/M_0), \qquad X = \text{ the image of } M_1 \text{ in } \cB, 
    \]
    where $\Idem$ means taking idempotent completion. 
    Then $\cB$ is smooth, $X \neq 0$ has $\End_\cB(X) = k$, and $
    \cB^!(X, -) = 0$.  
    Thus by Proposition~\ref{universal ff lcy}, the relative Calabi-Yau completion of the functor $k \xrightarrow{X} \cB$, 
    \[
    \cX := \Perf(k[z]) = \Perf(\Pi_1(k)) \longrightarrow \cS := \Perf(\Pi_2(\cB, k)), \qquad |z| = 0, 
    \]
    is a fully faithful nonzero left Calabi-Yau morphism. 
    
    In fact, $\cS$ is a localization of $\Perf(\Gamma)$, for a finite dimensional dg algebra $\Gamma$ with $\Perf(\Gamma)$ smooth and proper, with kernel generated by a single object $Z$ (Proposition \ref{prop: compactification}).  By Proposition \ref{cogeneration versus nondegeneracy}, $\End_\Gamma(Z)$ is another proper dg algebra which does not cogenerate.   
\end{example}

Our interest in Nonjecture \ref{proper dgas cogenerate} arises originally from Arnol'd's chord conjecture \cite{Arnold} in symplectic topology, which asserts that any compact Legendrian in any compact contact manifold has a Reeb chord --- had the Nonjecture been true, we could establish the existence of chords for contact manifolds which bound Weinstein domains.  We sketch an account of this in Appendix \ref{sec: chord}.  

In this context we note also that our examples here seem all closely related to algebras of differential operators.  We recall that, as noted in \cite{ganatra-nonproper}, these have a property obstructing an algebraic approach to a closely related conjecture in symplectic topology (existence of Reeb orbits  \cite{weinstein-orbit}).  Indeed, the algebra of differential operators $\mathcal{D}_X$ on a smooth affine algebraic variety $X$ is smooth, Calabi-Yau of dimension $2 \dim X$, obviously not proper, and yet, per \cite[Theorem 2]{Wodzicki1987}, 
has finite dimensional Hochschild homology $
HH_i(\mathcal{D}_X) \cong H_{\mathrm{dR}}^{2\dim X-i}(X)$.


\vspace{2mm}
\noindent {\bf{Acknowledgements.}} 
M. H. and V. S.  are supported by Villum Fonden Villum Investigator 37814. 

\section{A Cogeneration result: Proof of Proposition \ref{prop: bounded below}} \label{sec: bounded below}

Let $A$ be a dg algebra.  In this paper,
all modules are right modules, and we always use cohomological grading. We say, $A$ is \emph{coconnective} $H^i (A) = 0$ for $ i < 0$, and \emph{connective} if $H^i(A) = 0$ for $i>0$. 

Recall \cite[B.1]{Drinfeld-dg} that a dg $A$-module $F$ is \emph{semi-free} if $F = \colim_j F_{(j)}$ for a sequence of dg submodules $0 = F_{(0)} \subset F_{(1)} \subset \cdots$, where each $F_{(j+1)}$ is obtained from $F_{(j)}$ by adjoining free homogeneous generators $v$ with $dv \in F_{(j)}$; that is, $F_{(j+1)} = F_{(j)} \oplus (W_j \otimes A)$ as graded $A$-modules, where $W_j$ is the span of the new generators.  Thus $F = V \otimes A$ as a graded $A$-module, where $V = \bigoplus_j W_j$.  A \emph{semi-free resolution} of $M$ is a quasi-isomorphism $F \to M$ with $F$ semi-free; for such $F$, the complex $\Hom_A(F, -)$ computes the derived $\Hom$ \cite[B.2]{Drinfeld-dg}.

\begin{definition}
    Let $A$ be a dg algebra and let $M$ be an $A$-module with $H^*(M) \ne 0$ bounded below; let $
n = \min\{r : H^r(M) \ne 0\}$.
    We say a semi-free resolution $f : F \to M$ {\em has visible bottom} if: 
    all  generators have degree $\ge n$ and some closed degree $n$ generator $v_0 \in F_{(1)}$ represents a nonzero class in $H^n(M)$.
\end{definition}

\begin{lemma}
    Let $A$ be a dga with $H^*(A)$ bounded, and $N = \max\{r : H^r(A) \ne 0\}$. 
    Let $M\in \Mod_A$ with $H^*(M) \ne 0$ bounded below, and $n = \min\{r : H^r(M) \ne 0\}$. 
    Assume $M$ admits a semi-free resolution with visible bottom. 
    Then there is a degree $N-n$ map $\phi \in \Hom_A(F, A)$ that induces a nontrivial map on $H^{*}(M) \to H^{*+N-n}(A)$. In particular, $\Hom_A(M, A) \not\simeq 0$. 
\end{lemma}
\begin{proof}
    Choose a cocycle $a \in A^N$ with $[a] \ne 0$.  We construct $\phi$ with $\phi(v_0) = a$ 
    by induction. 

    Assume $F_{(j+1)} = F_{(j)} \oplus W_j \otimes A$ as a linear space. 
    On $F_{(1)} = W_0 \otimes A$ every generator is a cocycle; put $\phi(v_0) = a$, $\phi(v) = 0$ for the other generators of $W_0$, and extend $A$-linearly. 
    Suppose $\phi$ is defined on $F_{(j)}$ and let $v$ be a generator of $W_j$.  Then $dv \in F_{(j)}$, so $\phi(dv)$ is defined and it is a cocycle of degree $|v| + 1 + N - n \ge N + 1$.  As $H^{>N}(A) = 0$, there is $c \in A$ with $d_A c = \phi(dv)$; put $\phi(v) = (-1)^{N-n} c$.  Extending $A$-linearly defines a chain map $\phi: F_{(j+1)} \to A$.  By induction, we obtain a map $F_{(k)} \to A$ for each $k$, sending $v_0$ to $a$. Taking the colimit gives the desired map.
\end{proof}

The following is an application of a standard technique, see e.g. \cite{MaoWu2008}, although we did not find this precise assertion in the literature. 

\begin{lemma} \label{lem: semi-free resolution}
    Let $A$ be a coconnective dg algebra with $H^0(A) = k$ and $H^*(A)$ bounded above.  Then any $A$-module $M$ with $H^*(M) \ne 0$ bounded below admits a semi-free resolution with visible bottom.
\end{lemma}

\begin{proof}
    We construct semi-free modules $F_{(0)} \subset F_{(1)} \subset \cdots$ and compatible maps $f_j : F_{(j)} \to M$ such that $C_{(j)} = \cone(f_j)$ satisfies $H^{<n}(C_{(j)}) = 0$. 
    Moreover, the induced map $H^*(C_{(j)})\to H^*(C_{(j+1)})$ is always $0$. This implies that 
    $\cofib( \colim_j F_{(j)} \to M) = \colim (C_{(0)} \to C_{(1)} \to C_{(2)} \to \cdots ) \simeq 0$. Hence we can take $F = \colim_j F_{(j)}$ with the natural map $f: F \to M$. 
    
    Start with $F_{(0)} = 0$, so that $C_{(0)} = M$.  
    Assume $F_{(j)}$ and $f_j$ are given.
    For each $r\ge n$ choose cycles $(y_i, z_i) \in \cone(f_j) = M \oplus F_{(j)}[1]$ representing a basis of $H^r(C_{(j)})$, and for each of them adjoin to $F_{(j)}$ a free generator $v_i$ of degree $r$ with
    \[
    dv_i = z_i, \qquad f_{j+1}(v_i) = y_i.
    \]
    Since $dz_i = 0$ and $f_{j+1}(dv_i) = f_j(z_i) = dy_i$, this defines a dg module $F_{(j+1)}$ and a chain map $f_{j+1}$ extending $f_j$.  Let $W_j$ be the span of the new generators, a graded vector space concentrated in degrees $\ge n$, so that $F_{(j+1)} / F_{(j)} \cong W_j \otimes A$.
    The inclusion $F_{(j)} \subset F_{(j+1)}$ gives a short exact sequence
    \[
    0 \to C_{(j)} \to C_{(j+1)} \to (W_j \otimes A)[1] \to 0.
    \]
    The connecting map $\partial : H^r(W_j \otimes A) \to H^r(C_{(j)})$ sends $[v_i \otimes 1]$ to $[(y_i, z_i)]$, up to sign. By our construction, it is always a surjection. Hence $H^*(C_{(j)}) \to H^*(C_{(j+1)} )$ is zero. 

    Note that $H^n(W_j\otimes A)=W_j^n$, with basis the classes $[v_i\otimes1]$ of the generators of degree $n$. These were adjoined one for each element of a basis $[(y_i,z_i)]$ of $H^n(C_{(j)})$, and $\partial[v_i\otimes1]=\pm[(y_i,z_i)]$, so $\partial$ is an isomorphism in degree $n$, which implies that $H^{<n} (C_{(j+1)}) = 0$ (this is where we use $H^0(A) = k$). Therefore, $F_{(j+1)}$ satisfies all the prescribed conditions.

    Finally, the generators of degree $n$ adjoined at the first stage are indexed by a basis of $H^n(C_{(0)}) = H^n(M) \ne 0$; let $v_0$ be one of them.  Then $f(v_0)$ represents a nonzero class in $H^n(M)$.
\end{proof}

Proposition  \ref{prop: bounded below} now follows by combining the above two results. 

\begin{remark}
    The hypothesis $H^0(A) = k$ is  used in the proof of Lemma \ref{lem: semi-free resolution}  to ensure that no generators of degree $< n$ are needed.
    The  dg algebra in Example \ref{dga which does not cogenerate} is coconnective, has $H^0(A) = k \oplus k$, the module $M$ is bounded, and $\Hom_A(M, A) = 0$, so the hypothesis of Lemma \ref{lem: semi-free resolution} cannot be weakened even to asking $H^0(A)$ semisimple.
\end{remark}

\section{Example \ref{dga which does not cogenerate}} \label{sec: example dga}


Let $\Lambda$ be a dg algebra and $V$ a dg $\Lambda$-module.  Put
\[
V^* = \Hom_\Lambda(V, \Lambda), \qquad P = \Lambda \oplus V, \qquad A = \End_\Lambda(P),
\]
and let $e \in A$ be the projection onto $\Lambda$, so that $eAe = \Lambda$, $Ae = P$ and $eA = \Hom_\Lambda(P, \Lambda)$, with their indicated bimodule structures.  Put
\[
M = \cofib\left( Ae \otimes_{eAe} eA \to A \right) \in \Mod_A.
\]

\begin{lemma} \label{lem: general noncogeneration} We have the following
    \begin{enumerate}
        \item As a complex, $A \simeq \Lambda \oplus V \oplus V^* \oplus \End_\Lambda(V)$.  In particular $A$ is proper if and only if $\Lambda$, $V$, $V^*$ and $\End_\Lambda(V)$ have finite dimensional cohomology.
        \item $\Hom_A(M, A) = 0$.
        \item $M \simeq \cofib\left( \mathrm{ev} : V \otimes_\Lambda V^* \to \End_\Lambda(V) \right)$, so $M \ne 0$ if and only if $V$ is not perfect.
    \end{enumerate}
    Consequently, if $\Lambda$ is proper and $V$ is a $\Lambda$-module which is not perfect, but for which $V$, $\Hom_\Lambda(V, \Lambda)$ and $\End_\Lambda(V)$ have finite dimensional cohomology, then $A = \End_\Lambda(\Lambda \oplus V)$ is a proper dg algebra which does not cogenerate $\Mod_A$.
\end{lemma}




\begin{proof}
    (1) $\End_\Lambda(\Lambda \oplus V)$ is the direct sum of $\End_\Lambda(\Lambda) = \Lambda$, $\Hom_\Lambda(\Lambda, V) = V$, $\Hom_\Lambda(V, \Lambda) = V^*$ and $\End_\Lambda(V)$.

    (2) By the tensor-hom adjunction,
    \[
    \Hom_A(M, A) \simeq \Hom_A(M, \Hom_\Lambda(P, P)) \simeq \Hom_\Lambda(M \otimes_A P, P).
    \]
    Now $eA \otimes_A P = eP = \Lambda$, so
    \[
    (Ae \otimes_{eAe} eA) \otimes_A P = Ae \otimes_{eAe} eP = P \otimes_\Lambda \Lambda = P,
    \]
    and the map to $A \otimes_A P = P$ is the identity.  So $M \otimes_A P = 0$, and $\Hom_A(M, A) = 0$.

    (3) Since $Ae = \Lambda \oplus V$ and $eA = \Lambda \oplus V^*$ as right and left $\Lambda$-modules, respectively,
    \[
    Ae \otimes_{eAe} eA = (\Lambda \oplus V) \otimes_\Lambda (\Lambda \oplus V^*) = \Lambda \oplus V \oplus V^* \oplus (V \otimes_\Lambda V^*),
    \]
    and under (1) the map $Ae \otimes_{eAe} eA \to A$ is the identity on the first three summands and $\mathrm{ev}$ on the last.  Hence $M \simeq \cofib(\mathrm{ev})$.
\end{proof}

We now check the hypotheses of Lemma \ref{lem: general noncogeneration} for 
Example \ref{dga which does not cogenerate}. 

Recall that $\Lambda = k\langle x, y \rangle/(x^2, y^2, xy - qyx)$, $|x| = |y| = 2$, and $V = (\Lambda/(x+y)\Lambda)[-1]$.  Here $\Lambda$ is four dimensional, hence proper, and $V$ is two dimensional: explicitly,
\[
V = ku \oplus kv, \qquad |u| = 1, \quad |v| = 3, \qquad ux = -v, \quad uy = v, \quad vx = vy = 0.
\]

We first construct a semi-free resolution of $V$.  A direct computation in $\Lambda$ shows that the right annihilator of $x + (-q)^{-n}y$ is $(x + (-q)^{-n-1}y)\Lambda$.  It follows that 
\begin{equation} \label{eq: resolution of V}
F = \bigoplus_{n \ge 0} e_n \Lambda, \qquad |e_n| = n+1, \qquad d e_0 = 0, \qquad d e_{n+1} = e_n(x + (-q)^{-n} y), 
\end{equation}
with $F \to V$ given by $e_0 \mapsto u$ and $e_n \mapsto 0$ for $n \ge 1$, is a semi-free resolution of $V$: the kernel of $e_0 \Lambda \to V$ is $e_0(x + y)\Lambda$, and the kernel of $e_{n+1}\Lambda \to e_n \Lambda$ is $e_{n+1}(x + (-q)^{-n-1}y)\Lambda$.

Since $d(F) \subset F \cdot \Lambda^{>0}$, we have $F \otimes_\Lambda k = \bigoplus_n k e_n$ with zero differential, so $V \otimes_\Lambda k$ has infinite dimensional cohomology, and $V$ is not perfect.

Let us compute derived Hom using the semi-free resolution \eqref{eq: resolution of V}:
\begin{gather*}
V^* \simeq \Hom_\Lambda(F, \Lambda)
 = \prod_{n \ge 0} \Lambda e_n^*, \qquad |e_n^*| = -(n+1), \\
d(a\,e_n^*) = (-1)^{|a|-n}a(x+(-q)^{-n}y)e_{n+1}^*, \\
\End_\Lambda(V) \simeq \Hom_\Lambda(F,V)
 = \prod_{n \ge 0}(ku_n\oplus kv_n), \qquad |u_n|=-n,\quad |v_n|=2-n, \\
du_n=(-1)^{n+1}\big((-q)^{-n}-1\big)v_{n+1}, \qquad dv_n=0.
\end{gather*}
In the first complex, right multiplication by $x + (-q)^{-n} y$ is injective on $\Lambda^0$, has one dimensional kernel $k(x + (-q)^{1-n} y)$ on $\Lambda^2$, and is surjective onto $\Lambda^4$; and, for $n \ge 1$, $x + (-q)^{1-n} y$ is, up to sign, the image of $1 \in \Lambda^0 e_{n-1}^*$.  Hence the complex is exact except at $n = 0$, and
\[
H^*(V^*) = k[-1] \oplus k[-3], \qquad \text{generated by } w := (x - qy) e_0^*, \quad z := xy\, e_0^*. 
\]
In the second complex, $(-q)^{-n} \ne 1$ for $n \ge 1$ (as $q$ is not a root of unity), so 
\[
H^* \End_\Lambda(V) = k u_0 \oplus k v_1 \oplus k v_0 = k \oplus k[-1] \oplus k[-2]. 
\]

So Lemma \ref{lem: general noncogeneration} applies: $A$ is proper and does not cogenerate.  More precisely, by part (1) of the lemma,
\[
H^*(A) = k^2 \oplus k^3[-1] \oplus k^3[-2] \oplus k^2[-3] \oplus k[-4], \qquad H^0(A) = k \oplus k,
\]
so $A$ is coconnective with $\dim_k H^*(A) = 11$, as claimed.

Finally we compute $H^*(M)$, which shows that the module witnessing the failure of cogeneration has finite dimensional cohomology.  To compute the derived tensor product, use the two-dimensional left module of $\Lambda$-linear maps $V\to\Lambda$
\[
D=kw\oplus kz,
\]
where $w(u)=x-qy$, $w(v)=xy$, $z(u)=xy$, and $z(v)=0$.
Precomposition with $F\to V$ identifies these maps with the two cocycles above; hence
$D\to\Hom_\Lambda(F,\Lambda)$ is a quasi-isomorphism of left $\Lambda$-modules.  Its left action is
\[
xw=-qz,\qquad yw=q^{-1}z,\qquad xz=yz=0.
\]
Since $F$ is semi-free, $V\otimes_\Lambda V^*$ is computed by $F\otimes_\Lambda D$.
Writing $w_m=e_m\otimes w$ and $z_m=e_m\otimes z$, this complex has
\[
|w_m|=m+2,\quad |z_m|=m+4,\quad
dw_0=dz_m=0,\quad
dw_m=-\big((-q)^{-m}+q\big)z_{m-1}\quad(m\ge1).
\]
Each displayed coefficient is nonzero: otherwise $(-q)^{m+1}=1$, making $q$ a root of unity.
Thus the pairs $(w_m,z_{m-1})$, $m\ge1$, are contractible, leaving only $w_0$.  In particular
\begin{equation} \label{eq: V tensor V dual}
H^*(V \otimes_\Lambda V^*) = k w_0 = k[-2].
\end{equation}

By \eqref{eq: V tensor V dual} the source of $\mathrm{ev}: V \otimes_\Lambda V^* \to \End_\Lambda(V)$ is concentrated in degree $2$, and the generator $w_0 = e_0 \otimes (x-qy)e_0^*$ is sent by $\mathrm{ev}$ to the map $F \to V$, $e_0 b \mapsto u(x - qy)b$, $e_n \mapsto 0$ $(n \ge 1)$, which is $-(1+q) v_0 \ne 0$ in $H^2 \End_\Lambda(V)$.  Hence 
\[
H^*(M) = k \oplus k[-1]. 
\]
In particular $M$ is a nonzero module with finite dimensional cohomology and $\Hom_A(M, A) = 0$, which disproves Nonjecture \ref{proper dgas cogenerate}.

\begin{remark}
    In our example, our dga $A$ is coconnective. We are not aware of a proper \textit{connective} dg algebra that does not cogenerate its module category. 
\end{remark}

\section{Cogeneration and Nondegeneracy: Proof of Proposition \ref{cogeneration versus nondegeneracy}} \label{sec: proof of lemma}

We work over a field $k$, and write $(-)^\vee$ for the $k$-linear dual.  For a dg category $\cB$, we write $\Mod_\cB$ and ${}_\cB\Mod$ for the dg derived categories of right and left $\cB$-modules, i.e. of dg functors $\cB^{\op} \to \Mod_k$ and $\cB \to \Mod_k$.  For an object $x \in \cB$, we denote the Yoneda modules by
\[
h_x = \cB(-, x) \in \Mod_\cB, \qquad {}_x h = \cB(x, -) \in {}_\cB \Mod.
\]
An $(\cA, \cB)$-bimodule $M \in {}_\cA\Mod_\cB$ is a dg functor $M: \cA \otimes \cB^{\op} \to \Mod_k$, so that $M(a, -) \in \Mod_\cB$ and $M(-, b) \in {}_\cA\Mod$.  The diagonal bimodule of $\cB$ is $\cB_\Delta(x, y) = \cB(y, x)$, and $\cB^e = \cB \otimes \cB^{\op}$.

\begin{lemma}[{\cite[C.4 and C.7]{Drinfeld-dg}}] \label{bimodule yoneda}
    For a bimodule $M \in {}_\cA\Mod_\cB$ and objects $a \in \cA$, $b \in \cB$,
    \[
    h_a \otimes_\cA M \simeq M(a, -), \qquad M \otimes_\cB {}_b h \simeq M(-, b).
    \]
    In particular $h_b \otimes_\cB \cB_\Delta \simeq h_b$.
\end{lemma}

A dg category $\cB$ is {\em smooth} if $\cB_\Delta$ is perfect over $\cB^e$, and {\em proper} if $\cB(x, y)$ is a perfect complex over $k$ for all $x, y$.  If $\cB$ is smooth, the {\em inverse dualizing bimodule} $\cB^! = \Hom_{\cB^e}(\cB_\Delta, \cB^e)$ is again perfect, with $(\cB^!)^! \simeq \cB_\Delta$, and we write
\[
\Phi_\cB = (-) \otimes_\cB \cB^! : \Mod_\cB \to \Mod_\cB,
\]
so that $\Phi_\cB(h_b) = \cB^!(b, -)$ by Lemma \ref{bimodule yoneda}.  A module $M \in \Mod_\cB$ is {\em pseudo-perfect} if $M(b)$ is a perfect complex over $k$ for all $b$; if $\cB$ is proper, every $h_b$ is pseudo-perfect.

We will need the following lemma of Keller:

\begin{lemma}[{\cite[Lemma 3.4]{keller2009deformed}}] \label{lem:keller_duality}
    Let $\cB$ be smooth, $M \in \Mod_\cB$ pseudo-perfect, and $L \in \Mod_\cB$.  Then
    \[
    \Hom_{\Mod_\cB}(L \otimes_\cB \cB^!, M) \simeq \Hom_{\Mod_\cB}(M, L)^\vee.
    \]
\end{lemma}
\begin{proof}
    Since $M$ is pseudo-perfect, $\Hom_k(M, L) \simeq M^\vee \otimes_k L$ as bimodules.  Hence, using smoothness in the second step,
    \[
    \Hom_{\Mod_\cB}(M, L) \simeq \Hom_{\cB^e}(\cB_\Delta, \Hom_k(M, L)) \simeq \cB^! \otimes_{\cB^e} (M^\vee \otimes_k L) \simeq L \otimes_\cB \cB^! \otimes_\cB M^\vee.
    \]
    Taking the linear dual and using the tensor-hom adjunction, $\Hom_k(L \otimes_\cB \cB^! \otimes_\cB M^\vee, k) \simeq \Hom_{\Mod_\cB}(L \otimes_\cB \cB^!, M)$.
\end{proof}

Let $\cC$ be a smooth dg category, $\pi: \cC \to \cS = \cC/\cK$ be a dg quotient, with $\cK \subset \cC$ generated by a set of objects. There is a pair of adjoint functors:
\[
\begin{tikzcd}[column sep=large]
\Mod_{\cC}
  \arrow[r, bend left=20, "\pi"]
&
\Mod_\cS
  \arrow[l, bend left=20, "\iota"]
\arrow[phantom, from=1-1, to=1-2, "\dashv"{rotate=-90}]
\end{tikzcd}.
\]
We recall the following lemma:
\begin{lemma}[{\cite[Proposition 3.10]{keller2009deformed}; see also \cite[Lemma 2.34]{Christ26}}] \label{lem:localization and dualizing}
    In this situation, if $\cC$ is smooth then 
    \[
     \Phi_\cS \simeq \pi \circ \Phi_\cC \circ \iota.
    \]
\end{lemma}

Now we can give a proof of Proposition~\ref{cogeneration versus nondegeneracy}. 
\begin{proof}[Proof of Proposition \ref{cogeneration versus nondegeneracy}]
    Let $M \in \Mod_\cS$ with $M \otimes_\cS \cS^! = 0$. Then by Lemma~\ref{lem:localization and dualizing} we have $\pi \circ \Phi_\cC \circ \iota (M) = 0$, which implies that $\Phi_\cC \circ \iota (M) \in \Loc(\cK) \simeq \Mod_\cK$. 

    Let $K$ be an object in $\cK$. 
    As $\cC$ is proper, $h_K$ is pseudo-perfect, so by Lemma \ref{lem:keller_duality}
    \begin{equation} \label{keller duality applied}
    \Hom_{\Mod_\cC}(\Phi_\cC \circ \iota (M), h_K) \simeq \Hom_{\Mod_\cC}(h_K, \iota (M))^\vee \simeq \Hom_{\Mod_\cS}(\pi ( h_K) , M)^\vee \simeq 0 .
    \end{equation}

    Thus $\Phi_\cC \circ \iota (M)$ is an object of $\Loc(\cK) \simeq \Mod_\cK$ with $\Hom_{\Mod_\cC}(\Phi_\cC \circ \iota (M), h_K) = 0$ for every $K \in \cK$, hence for every compact object of $\Mod_\cK$.  By hypothesis $\Mod_\cK$ is cogenerated by its compact objects, so $\Phi_\cC \circ \iota (M)= 0$. Since $\cC$ is smooth and proper, $\Phi_\cC$ is an autoequivalence of $\Mod_\cC$, with inverse the Serre functor $(-) \otimes_\cC \cC^\vee$. Hence $\iota (M) = 0$, and $M = 0$, as $\iota$ is fully faithful.
\end{proof}

The following result will later be useful:

\begin{lemma} \label{lem: general degenerate}
    Let $\cC$ be a smooth dg category and $Y \in \cC$ a nonzero object such that $\Phi_\cC(Y)$ is perfect and
    \[
    \Hom_\cC(\Phi_\cC(Y), Y) = 0.
    \]
    Let $\cB$ be the idempotent completion of the dg quotient $\cC/\Phi_\cC(Y)$, and let $X$ be the image of $Y$ in $\cB$.  Then $\cB$ is smooth, $X$ is a nonzero compact object of $\cB$ with $\Hom_\cB(X, X) \simeq \Hom_\cC(Y, Y)$, and $\Phi_\cB(X) = 0$.  In particular the inverse dualizing bimodule of the smooth dg category $\cB$ is degenerate.
\end{lemma}
\begin{proof}
Put $K = \Phi_\cC(Y)$, and consider the adjoint pair
\[
\begin{tikzcd}[column sep=large]
\Mod_{\cC}
  \arrow[r, bend left=20, "\pi"]
&
\Mod_\cB
\arrow[l, bend left=20, "\iota"]
\arrow[phantom, from=1-1, to=1-2, "\dashv"{rotate=-90}]
\end{tikzcd}.
\]
Note that $\iota$ is fully faithful with image $K^\perp$, and $\pi$ preserves compact objects, since $K$ is compact.  By Lemma~\ref{lem:localization and dualizing}, $\Phi_\cB=\pi\Phi_\cC\iota$.
By hypothesis $Y\in K^\perp$, so $\iota X\simeq Y$.  Thus $X\neq 0$,
\[
\Hom_\cB(X,X)\simeq \Hom_\cC(Y, \iota \pi Y) \simeq \Hom_\cC(Y,Y),
\]
and $X$ is compact because $Y$ is.  Finally,
$
\Phi_\cB(X)\simeq\pi\Phi_\cC(Y)=\pi(K)=0.
$
\end{proof}

The hypothesis of Lemma \ref{lem: general degenerate} never holds when $\cC$ is smooth and proper: then $\Phi_\cC$ is the inverse of the Serre functor, and Serre duality gives $\Hom_\cC(\Phi_\cC(Y), Y)  \simeq \Hom_\cC(Y, Y)^\vee \ne 0$.

\section{On Fully Faithful Left Calabi-Yau Morphisms: Proof of Proposition \ref{universal ff lcy}}
\label{sec:ff lcy}

Let $f: \cA \to \cB$ be a dg functor between smooth dg categories.  We write ${}_\cA\cB_\cB$ and ${}_\cB\cB_\cA$ for the one-sided pullbacks of the diagonal bimodule ${}_\cB\cB_\cB$.  There is an adjoint pair
\[
\begin{tikzcd}[column sep=large]
{}_\cA \Mod_\cA 
 \arrow[r, bend left=20, "f_!"]
& {}_\cB \Mod_\cB 
\arrow[l, bend left=20, "f^*"]
\arrow[phantom, from=1-1, to=1-2, "\dashv"{rotate=-90}]
\end{tikzcd}
\]
where $f_!(M) = \cB \otimes_\cA M \otimes_\cA \cB := {}_\cB\cB_\cA \otimes_\cA M \otimes_\cA {}_\cA\cB_\cB$ and $f^*$ is restriction of bimodules.  In particular $f_!(\cA_\Delta) = \cB \otimes_\cA \cB$.  Let $c$ be the composite
\[
c: f_!(\cA_\Delta) \longrightarrow f_!f^*\cB_\Delta \longrightarrow \cB_\Delta.
\]
If $f$ is fully faithful, the first map in the definition of $c$ is an equivalence, so $c$ is the counit. 
Since $\cA$ is smooth, $f_!(\cA_\Delta)^! \simeq f_!(\cA^!)$ \cite[Lemma 2.9]{keller2009deformed}, so the dual of $c$ is a map $c^!: \cB^! \to f_!(\cA^!)$.

\begin{lemma} \label{lem: ff tensor}
    If $f: \cA \to \cB$ is fully faithful, then $h_{f(a)} \otimes_\cB \cB \otimes_\cA \cB \simeq h_{f(a)}$, and under this identification $h_{f(a)} \otimes c$ is the identity.
\end{lemma}
\begin{proof}
    We have
    \[
    h_{f(a)} \otimes_\cB \cB \otimes_\cA \cB \simeq \cB(f(-), f(a)) \otimes_\cA \cB \simeq h_a \otimes_\cA \cB \simeq \cB(-, f(a)).
    \]
    Here the first and third equivalences are Lemma \ref{bimodule yoneda}, and the second is full faithfulness.
\end{proof}

We recall from \cite{brav-dyckerhoff-1, brav-dyckerhoff-2} the definition of relative left Calabi-Yau structure.  For smooth $\cB$ one has $HH(\cB) = \cB_\Delta \otimes_{\cB^e} \cB_\Delta \simeq \Hom_{\cB^e}(\cB^!, \cB_\Delta)$.  Thus a class in $HH_n(f) = H_n\, \cone\left( HH(\cA) \to HH(\cB) \right)$ amounts to a map $\cA^![n-1] \to \cA_\Delta$ together with a null-homotopy of the composite
\[
\cB^![n-1] \xrightarrow{\ c^!\ } f_!(\cA^!)[n-1] \longrightarrow f_!(\cA_\Delta) \xrightarrow{\ c\ } \cB_\Delta,
\]
and hence induces a map of exact triangles
\begin{equation} \label{eq: left CY nondegeneracy}
\begin{tikzcd}
    \cB^![n-1] \ar[r, "c^!"] \ar[d]
    & f_!(\cA^!)[n-1] \ar[r] \ar[d]
    & \cofib(c^!)[n-1] \ar[r] \ar[d]
    & \cB^![n] \ar[d] \\
    \fib(c) \ar[r]
    & f_!(\cA_\Delta) \ar[r, "c"]
    & \cB_\Delta \ar[r]
    & \cofib(c)
\end{tikzcd}
\end{equation}
whose second vertical map is $f_!$ of the map $\cA^![n-1] \to \cA_\Delta$.

\begin{definition}
    Let $f: \cA \to \cB$ be a dg functor between smooth dg categories.  A {\em relative left Calabi-Yau structure} of dimension $n$ on $f$ is a class in the relative negative cyclic homology
    \[
    HC^-_n(f) = H_n\, \cone\left( CC(\cA)^{S^1} \to CC(\cB)^{S^1} \right)
    \]
    such that, for the underlying class in $HH_n(f)$, the map $\cA^![n-1] \to \cA_\Delta$ and all vertical maps in \eqref{eq: left CY nondegeneracy} are isomorphisms.
\end{definition}

\begin{lemma} \label{lem: kernel of IDB}
    If $f: \cA \to \cB$ is a fully faithful left Calabi-Yau morphism of dimension $n$, then for every $a \in \cA$,
    \[
    \cB^!(f(a), -) = 0 = \cB^!(-, f(a)).
    \]
    More precisely, let $q: \cB \to \cB/\cA$ be the quotient, with $q_! \dashv q^*$ on module categories. Then
    \[
    \Phi_\cB[n] \simeq q^*q_!, \qquad
    \ker(\Phi_\cB)=\Loc\{h_{f(a)}:a\in\cA\}\subset\Mod_\cB.
    \]
\end{lemma}
\begin{proof}
    For any $M \in \Mod_\cB$, we tensor the bottom row of \eqref{eq: left CY nondegeneracy} by $M$ and use the last vertical isomorphism to obtain an exact triangle
    \[
    M \otimes_\cB \cB \otimes_\cA \cB \longrightarrow M \otimes_\cB \cB_\Delta \longrightarrow M \otimes_\cB \cB^![n] \longrightarrow
    \]
    The first map is the counit $(M|_\cA)\otimes_\cA\cB\to M$. Since $f$ is fully faithful, this is the localization triangle for $q$, whose cofiber is $q^*q_!M$. Thus $\Phi_\cB[n]\simeq q^*q_!$, and the kernel assertion follows from the full faithfulness of $q^*$ and $\ker(q_!)=\Loc\{h_{f(a)}:a\in\cA\}$.
\end{proof}

To get a sort of converse, we turn to Calabi-Yau completions \cite{keller2009deformed, Yeung}.
For a dg category $\cA$ and $N \in {}_\cA \Mod_\cA$, the \emph{tensor dg category} $T_\cA(N)$ has the same objects as $\cA$ and morphisms are defined by the tensor bimodule
\[
T_\cA(N) (a, b) = 
\big( \cA_\Delta\oplus N\oplus N^{\otimes 2}\oplus\cdots \big) (b, a)
\]
with natural compositions \cite{keller2009deformed}.

For $f: \cA \to \cB$ between smooth dg categories, let $f_! \dashv f^*$, $c$ and $c^!$ be as in Section \ref{sec: proof of lemma}.   Put
\[
\Theta_{\cB/\cA}=\cofib\left(c^!:\cB^!\longrightarrow f_!(\cA^!)\right)[-1],
\qquad
\Pi_n(\cB,\cA)=T_\cB\left(\Theta_{\cB/\cA}[n-1]\right).
\]
The canonical map $f_!(\cA^![n-2])\to\Theta_{\cB/\cA}[n-1]$ induces the relative completion functor
\[
\Pi_{n-1}(\cA)=T_\cA(\cA^![n-2])\longrightarrow\Pi_n(\cB,\cA).
\]

\begin{lemma} \label{lem: key}
    Let $f:\cA\to\cB$ be a fully faithful functor between smooth dg categories. If $\cB^!(f(a),-)=0$ for every $a\in\cA$, then
    $
    \Pi_{n-1}(\cA)\longrightarrow\Pi_n(\cB,\cA)
    $
    is fully faithful, as is the induced functor on perfect categories. The same holds with $\cB^!(-,f(a))=0$ in place of $\cB^!(f(a),-)=0$.
\end{lemma}
\begin{proof}
Since representable right $\cA$-modules generate $\Mod_\cA$ and tensor products preserve colimits, the hypothesis gives
\[
(M\otimes_\cA\cB)\otimes_\cB\cB^!=0
\qquad(M\in\Mod_\cA).
\]
Full faithfulness identifies $f^*\cB_\Delta$ with $\cA_\Delta$, so
\begin{align*}
(M\otimes_\cA\cB)\otimes_\cB f_!(\cA^![n-2])
&\simeq M\otimes_\cA  \cB \otimes_\cA \cA^![n-2] \otimes_\cA \cB \\
&\simeq M\otimes_\cA f^*\cB_\Delta \otimes_\cA \cA^![n-2] \otimes_\cA \cB\\
&\simeq (M\otimes_\cA\cA^![n-2])\otimes_\cA\cB.
\end{align*}
Tensoring the triangle
\begin{equation}
    \label{eq:triangle}
    \cB^![n-2]\longrightarrow f_!(\cA^![n-2])
    \longrightarrow\Theta_{\cB/\cA}[n-1]
\end{equation}
with $M\otimes_\cA\cB$ therefore gives an equivalence
\[
(M\otimes_\cA\cA^![n-2])\otimes_\cA\cB
\xrightarrow{\sim}(M\otimes_\cA\cB)\otimes_\cB\Theta_{\cB/\cA}[n-1].
\]
Iteratively tensoring with~\eqref{eq:triangle} gives 
\[
\bigl(M\otimes_\cA(\cA^![n-2])^{\otimes_\cA m}\bigr)\otimes_\cA\cB
\xrightarrow{\sim}(M\otimes_\cA\cB)\otimes_\cB\bigl(\Theta_{\cB/\cA}[n-1]\bigr)^{\otimes_\cB m}
\qquad(m\ge0).
\]
Taking $M=h_a$ and evaluating at $f(a')$ gives equivalences
\[
\bigl(\cA^![n-2]\bigr)^{\otimes_\cA m}(a,a')
\xrightarrow{\sim}\bigl(\Theta_{\cB/\cA}[n-1]\bigr)^{\otimes_\cB m}(f(a),f(a'))
\qquad(m\ge0).
\]
This shows that the map $\Pi_{n-1} (\cA) \to \Pi_n (\cB, \cA) $ is fully faithful.

The other vanishing condition is treated in the same way using left modules.
\end{proof}

\section{Example \ref{eg:nontrivial ff lcy}} \label{sec: example lcy}

In this section, we ask $k$ to be a field with characteristic $0$. We prove the claims made in Example \ref{eg:nontrivial ff lcy}.   We use the conventions of Sections \ref{sec: proof of lemma} and \ref{sec:ff lcy}; in particular $\Phi_\cB = (-) \otimes_\cB \cB^!$ for a smooth dg category $\cB$.  

Recall
\[
    R = k\langle x, y, \delta \rangle / ( [x, y], [\delta, x] - 1, [\delta, y] + y ),
     \qquad M_\lambda = R/(\delta - \lambda)R \quad (\lambda \in k).
\]
In other words, $R = k[x, y]\langle \delta \rangle$ with $\delta = \partial_x - y\, \partial_y$.  It has the PBW bases $\{x^a y^b \delta^c\}$ and $\{\delta^c x^a y^b\}$.

Since $R$ is a domain, there is a resolution $0 \to R \xrightarrow{(\delta - \lambda)\cdot} R \to M_\lambda \to 0$. Hence $M_\lambda$ is perfect.  The PBW basis $\{\delta^c x^a y^b\}$ identifies $M_\lambda = k[x, y]$ as a vector space, with right action
\begin{equation} \label{eq: M action}
f \cdot x = fx, \qquad f \cdot y = f y, \qquad f \cdot \delta = \lambda f - \partial_x f + y \, \partial_y f.
\end{equation}

For a right $R$-module $M$, and an endomorphism $\nu \in \End(R)$, denote by
$M^\nu$ the right module twisted by $\nu$, whose action is defined by
\[
 m * r = m \cdot \nu(r).
\]

\subsection{A smooth category with degenerate inverse dualizing bimodule}

We consider the following automorphism of $R$:
\[
\nu(x) = x, \qquad \nu (y) = y, \qquad \nu(\delta) = \delta - 1. 
\]

\begin{lemma}[See also~\cite{LiuMa2020}]
    \label{prop: nakayama}
    The algebra $R$ is smooth, and $R^! \simeq R^\nu[-3]$, where $R^\nu$ is $R$ with its right action twisted by $\nu$.  Consequently $\Phi_R(M_\lambda) \simeq M_{\lambda - 1}[-3]$.
\end{lemma}
\begin{proof}

For $a\in R$, put $D_a=a\otimes1-1\otimes a^{\op}\in R^e$.
On a bimodule this acts by $b\mapsto ab-ba$, so these are the differences
between the left and right actions.  They satisfy
$[D_a,D_b]=D_{[a,b]}$.  In particular, 
\[
[ D_\delta, D_x] = 0, \qquad [D_\delta, D_y] = -D_y, \qquad [D_x, D_y] = 0.
\]

Let $\mathfrak g$ be the Lie algebra with basis $e_x,e_y,e_\delta$ and brackets
$[e_\delta,e_x]=0$, $[e_\delta,e_y]=-e_y$, $[e_x,e_y]=0$. Then $e_i\mapsto D_i$
defines a Lie algebra map $\mathfrak g\to R^e$.
Then $R^e$ is free as a right $U(\mathfrak g)$-module, and $R^e\otimes_{U(\mathfrak g)}k \cong R$.
Applying $R^e\otimes_{U(\fg)}-$ to the Chevalley--Eilenberg resolution
$U(\fg)\otimes\bigwedge^\bullet\fg\to k$ gives a free resolution of $R$ as an $R^e$-module:
\begin{equation} \label{eq: bimodule resolution}
0\rightarrow R^e \otimes \bigwedge^3 \fg \xrightarrow{d_3} R^e \otimes \bigwedge^2 \fg \xrightarrow{d_2} R^e \otimes \fg \xrightarrow{d_1} R^e \xrightarrow{\mu} R \rightarrow 0.
\end{equation}
Thus $R$ is smooth. Here $\mu(a\otimes b^{\op})=ab$, and explicitly,
\begin{align*}
    d_1(e_i)&=D_i \qquad (i\in\{x,y,\delta\});\\
    d_2(e_i\wedge e_j)&=D_i e_j-D_j e_i-[e_i,e_j]
    \qquad (i,j\in\{x,y,\delta\});\\
    d_3(e_x\wedge e_y\wedge e_\delta)&=D_x(e_y\wedge e_\delta)-D_y(e_x\wedge e_\delta)
    +(D_\delta+1)(e_x\wedge e_y).
\end{align*}

Dualizing \eqref{eq: bimodule resolution} computes $R^!$. An explicit computation shows that the cohomology is concentrated in degree $3$, and
\[
\Ext^3_{R^e}(R,R^e)
=R^e\big/(D_xR^e+D_yR^e+(D_\delta+1)R^e).
\]
The assignment
\[
a\otimes b^{\op}\longmapsto b\,\nu(a)
\]
therefore induces an isomorphism from this quotient to $R^\nu$, as $R^e$-modules.  Hence $R^!\simeq R^\nu[-3]$.

Finally, $M_\lambda\otimes_R R^\nu\simeq M_\lambda^\nu
\simeq M_{\lambda-1}$ by \eqref{eq: M action}, which proves the last assertion.
\end{proof}

\begin{lemma} \label{prop: homs}
    $\Hom_R(M_\lambda, M_\mu) = k$, in degree $0$, if $\lambda - \mu \in \bZ_{\ge 0}$, and $\Hom_R(M_\lambda, M_\mu) = 0$ otherwise.
\end{lemma}
\begin{proof}
Applying $\Hom_R(-,M_\mu)$ to the resolution of $M_\lambda$ gives
\[
\left[M_\mu\xrightarrow{\cdot(\delta-\lambda)}M_\mu\right]
\]
in degrees $0,1$. By \eqref{eq: M action}, the differential is
\[
f\longmapsto(\mu-\lambda)f-\partial_xf+y\partial_yf.
\]
It preserves each summand $k[x]y^n$ and acts there as $(\mu+n-\lambda)-\partial_x$.
If $\mu+n-\lambda\ne0$, this is invertible, with inverse
\[
p\longmapsto\sum_{j\ge0}(\mu+n-\lambda)^{-j-1}\partial_x^jp,
\]
the sum being finite for each polynomial $p$. If $\mu+n-\lambda=0$, it is
$-\partial_x$, which is surjective with kernel $k$ in characteristic zero.
Thus the complex has cohomology $k$ in degree $0$ when
$\lambda-\mu\in\bZ_{\ge0}$, and is acyclic otherwise.
\end{proof}

\begin{proposition} \label{prop: X}
    Let $\cB = \Idem(\Perf(R)/M_0)$, and let $X$ be the image of $M_1$.  Then $\cB$ is smooth, $X$ is a nonzero compact object of $\cB$ with $\Hom_\cB(X, X) = k$, and $\Phi_\cB(X) = 0$.
\end{proposition}
\begin{proof}
By Lemma \ref{prop: nakayama}, $\Phi_R(M_1)\simeq M_0[-3]$, which is perfect.  By Lemma \ref{prop: homs}, $\Hom_R(M_0,M_1)=0$ and $\Hom_R(M_1,M_1)=k$.  So Lemma \ref{lem: general degenerate} applies to $\cC=\Perf(R)$ and $Y=M_1$.
\end{proof}

By Lemma \ref{lem: key}, the relative $2$-Calabi-Yau completion of $k \xrightarrow{X} \cB$,
    \[
    F: \cX := \Perf(k[z]) \longrightarrow \cS := \Perf(\Pi_2(\cB, k)), \qquad |z| = 0,
    \]
    is fully faithful. 

\begin{remark}[The idea behind the construction] \label{rem: idea}
Lemmas \ref{lem: general degenerate} and \ref{lem: key} reduce the construction of a fully faithful left Calabi--Yau functor with nonzero source to finding a smooth dg category $\cC$ and an object $Y$ with $\End_\cC(Y)=k$, $\Phi_\cC(Y)$ perfect, and $\Hom_\cC(\Phi_\cC(Y),Y)=0$; as noted above, $\cC$ cannot be proper.  
In order to check the hypotheses of Lemma \ref{lem: general degenerate} (to produce a category with degenerate dualizing bimodule), it is convenient to have $\cC=\Perf(A)$ for an algebra $A$ with $A^!\simeq A^\nu[-d]$ for some automorphism $\nu$: then $\Phi_A(M)=M^\nu[-d]$ is perfect, and the condition reads $\Hom_A(M^\nu,M)=0$.


For $M=A/aA$ with $a$ a nonzerodivisor, $\Hom_A(M,M)=[M\xrightarrow{\cdot a}M]$, so $\End_A(M)=k$ asks that right multiplication by $a$ be surjective on $M$ with one dimensional kernel.  

The basic such operator is $\partial_x$ on $k[x]$, which leads to the modules $A/(\delta-\lambda)A$ over the Weyl algebra $A=k\langle x,\delta\rangle/([\delta,x]-1)$.  But there $\nu=\mathrm{id}$, so $\Hom_A(M^\nu,M)\ne0$.  Adjoining a variable $y$ with $[x,y]=0$ and $[\delta,y]=-y$ produces the twist $\nu(\delta)=\delta-1$ (Lemma \ref{prop: nakayama}); the shift is the divergence of the vector field $\delta=\partial_x-y\,\partial_y$, whereas $\partial_x$ is divergence free.  Then $M_\lambda^\nu\simeq M_{\lambda-1}$, and $\Hom_R(M_{\lambda-1},M_\lambda)=0$ by Lemma \ref{prop: homs}.
\end{remark}

\begin{remark} Example~\ref{eg:nontrivial ff lcy} also provides an example of a left Calabi--Yau functor that is not spherical, since for a spherical functor $\cA \to \cB$ the inverse Serre functor has to be an equivalence. Such an example coming from geometry can be found in~\cite{Kuo-Li-CY}. 
\end{remark}

The rest of this section aims to find a categorical compactification of $\Pi_2(\cB, k)$, as prescribed in Example~\ref{eg:nontrivial ff lcy}. 

\subsection{A categorical compactification of $
    R = k\langle x, y, \delta \rangle / ( [x, y], [\delta, x] - 1, [\delta, y] + y )
    $}
Note that $R$ is not proper, but it is smooth and of finite global dimension since it is an Ore extension of $k[x, y]$. 

First let us recall some constructions from noncommutative projective geometry (see e.g.\ \cite{ArtinZhang}).  

Let $A=\bigoplus_{d\ge0}A_d$ be a connected graded algebra, that is, $A_0=k$, $A$ is generated by $A_1$, and each $A_d$ is finite dimensional; we assume $A$ right Noetherian.
Let $\gr A$ be the abelian category of finitely generated graded right $A$-modules and degree preserving maps, with shift $M(n)_d=M_{n+d}$, and let $\mathrm{tors}\subset\gr A$ be the Serre subcategory of finite dimensional modules.  Let $\qgr A=\gr A/\mathrm{tors}$ be the quotient category, let $\cO(n)$ be the image of $A(n)$ in $\qgr A$, and write $H^j(\cF)=\operatorname{Ext}^j_{\qgr A}(\cO,\cF)$ for $\cF\in\qgr A$.
For commutative $A$, Serre's theorem identifies $\mathrm{qgr}\,A$ with $\coh(\operatorname{Proj}A)$; in general, $\mathrm{qgr}\,A$ is the category of coherent sheaves on the noncommutative projective scheme $\operatorname{Proj}A$ of \cite{ArtinZhang}.

A connected graded algebra $A$ is Artin--Schelter Gorenstein of dimension $d$ and Gorenstein parameter $\ell$ \cite{ArtinSchelter} if $A$ has injective dimension $d$ as a left and as a right module over itself, and
\[
\operatorname{Ext}^i_A(k,A)\simeq\begin{cases}k(\ell),&i=d,\\0,&i\ne d,\end{cases}
\]
on both sides; it is Artin--Schelter regular if moreover $A$ has global dimension $d$.  In that case the minimal graded free resolution of $k$ has length $d$ and ends with $A(-\ell)$; for $A=k[x_0,\dots,x_{d-1}]$ it is the Koszul resolution, and $\ell=d$.  In general $\ell=d$ if and only if this resolution is linear, i.e. $A$ is Koszul, and a quadratic algebra with a PBW basis is Koszul \cite{Priddy}.  Noetherian Artin--Schelter regular algebras satisfy the condition $\chi$ of \cite{ArtinZhang}, and the cohomology of $\cO(n)$ is computed by local cohomology:
\[
H^0(\cO(n))=A_n,\qquad H^j(\cO(n))=0\ \ (0<j<d-1),\qquad H^{d-1}(\cO(n))=(A_{-n-\ell})^\vee .
\]
Thus $\mathrm{qgr}\,A$ behaves like $\coh(\mathbf P^{d-1})$, with $\cO(-\ell)$ in the role of the canonical bundle. We recall the following theorem:

\begin{theorem}[{\cite[Theorem 4.14]{MinamotoMori}}] \label{thm: MM}
Let $A$ be a right Noetherian Artin--Schelter regular algebra of dimension $d\ge1$ and Gorenstein parameter $\ell\ge1$.  Then $\bigoplus_{i=0}^{\ell-1}\cO(i)$ is a tilting object of $D^b(\mathrm{qgr}\,A)$, and its endomorphism algebra $\nabla A$ is finite dimensional of finite global dimension.  Hence $D^b(\mathrm{qgr}\,A)\simeq\Perf(\nabla A)$.
\end{theorem}

Ordering the summands as $\cO(\ell-1),\ldots,\cO$, it is the upper triangular matrix algebra with $(i,j)$ entry $A_{j-i}$, $0\le i\le j\le\ell-1$, and multiplication induced by that of $A$.


Now we are ready to construct a categorical compactification of $\Perf (R)$. 

\begin{proposition} \label{prop:cat cpt R}
Let $\widetilde R=k\langle t,x,y,\delta\rangle/
\bigl([t,x],[t,y],[t,\delta],[x,y],[\delta,x]-t^2,[\delta,y]+yt\bigr)$
be the homogenization of $R$. 
Let $\widehat\cR=D^b(\qgr\widetilde R)$ with its natural dg enhancement.  Then $\widehat\cR$ is smooth and proper, and there is a localization $q_R: \widehat\cR \to \Perf(R)$.  Moreover, $\ker (q_R)$ is generated by finitely many objects. 
\end{proposition}
\begin{proof}
The algebra $\widetilde R$ is a Noetherian domain with PBW basis $\{t^ax^by^c\delta^d\}$.
In particular, $t$ is a central non-zero-divisor and the polynomial ring $\widetilde R/(t)\simeq k[x,y,\delta]$ is of course Artin--Schelter regular of dimension $3$.
Hence $\widetilde R$ is Artin--Schelter regular of dimension $4$ by \cite[Proposition 2.1]{CKS}.
Its presentation is quadratic with a PBW basis, so $\widetilde R$ is Koszul, with Gorenstein parameter $4$.
Finally, $\widetilde R[t^{-1}]_0\simeq R$ via $x/t,\,y/t,\,\delta/t\mapsto x,y,\delta$, where $\widetilde R[t^{-1}]_0$ denotes the internal degree $0$ part of the localization $\widetilde R[t^{-1}]$, if we view $t,x,y,\delta$ as having degree one.

By Theorem \ref{thm: MM}, with $d=\ell=4$, $G=\bigoplus_{i=0}^3\cO(i)$ is a tilting object of $\widehat\cR$.
Hence $\widehat\cR\simeq\Perf(E)$ with $E=\End_{\widehat\cR}(G)=\nabla\widetilde R$ finite dimensional, of dimension $\sum_{0\le i\le j\le 3}\dim\widetilde R_{j-i}=56$, and of finite global dimension; in particular $\widehat\cR$ is smooth and proper.
The exact functor $N\mapsto N[t^{-1}]_0$ is a Serre quotient $\mathrm{qgr}\,\widetilde R\to\mathrm{mod}\,R$ whose kernel consists of the $t$-torsion objects.
By \cite{Miyachi} it induces a localization
\[
q_R : \widehat\cR \longrightarrow D^b(\mathrm{mod}\,R) = \Perf(R),
\]
whose kernel consists of the complexes with $t$-torsion cohomology.
A finitely generated $t$-torsion module has a finite filtration whose subquotients are modules over $\widetilde R/(t)=k[x,y,\delta]$, and $D^b(\mathrm{qgr}\,k[x,y,\delta])=D^b(\mathrm{coh}\,\mathbf P^2)$ is generated by $\cO,\cO(1),\cO(2)$ \cite{Beilinson}.
Hence
\[
\ker(q_R)=\thick(D_0,D_1,D_2),
\]
where $D_j$ is the image of $\cO_{\mathbf P^2}(j)$ under the pushforward from the hyperplane $t=0$ at infinity.
\end{proof}

\subsection{Categorical compactification and Calabi-Yau completion} \label{sec: compactification}

We first recall the notion of 
noncommutative $\mathbf P^1$-bundle \cite{VdBP1}.  For a dg category $\cA$ and an $\cA$-bimodule $N$, let $\mathbb P_\cA(N)$ be the dg category obtained by gluing two copies $\cA_0$, $\cA_1$ of $\cA$ along the bimodule $\cA_\Delta \oplus N$.  Its objects are $a_0$, $a_1$ for $a \in \cA$, with
\[
\Hom(a_0, a'_0) = \Hom(a_1, a'_1) = \cA(a, a'), \qquad \Hom(a_1, a'_0) = \cA(a, a') \oplus N(a', a), \qquad \Hom(a_0, a'_1) = 0,
\]
and composition given by the bimodule structure.  

Note that if $\cA$ is smooth and proper and $N$ is perfect, then the gluing bimodule $\cA_\Delta\oplus N$ is perfect, so by \cite{Orlov16},
$\mathbb P_\cA(N)$ is smooth and proper.

\begin{example}
For a dg algebra $E$ this is the category of modules over the triangular matrix algebra
\[
\mathbb P_E(N) = \begin{pmatrix} E & E \oplus N \\ 0 & E \end{pmatrix},
\]
with diagonal idempotents $e_0, e_1$ and $e_0\, \mathbb P_E(N)\, e_1 = E \oplus N$.
In particular, $\mathbb P_k(k)$ is the path algebra of the Kronecker quiver, whose module category Beilinson identified with  $D^b(\mathbf P^1)$. 
\end{example}

Write $u_a = (\mathrm{id}_a, 0) \in \Hom(a_1, a_0)$; these maps form a natural transformation $u$ from the inclusion of $\cA_1$ to the inclusion of $\cA_0$.  As one may expect, ``restricting to the vector bundle part of the projective bundle'' is a localization:  

\begin{lemma} \label{lem: P1 bundle}
 Let $\Phi: \mathbb P_\cA(N) \to T_\cA(N)$ be the dg functor which is the identity on both copies of $\cA$ and the inclusion $\cA_\Delta \oplus N \subset T_\cA(N)$ on $\Hom(a_1, a'_0)$.  Then $\Phi$ induces a localization $\Phi_!: \Perf(\mathbb P_\cA(N)) \to \Perf(T_\cA(N))$ with kernel $\thick\{\cofib(u_a) : a \in \cA\}$.  In particular, if $\cA = \thick(G)$, the kernel is $\thick(\cofib(u_G))$.
\end{lemma}
\begin{proof}
Write $T=T_\cA(N)$ and $B=\cA_\Delta\oplus N$.
A right $\mathbb P_\cA(N)$-module $F$ is the same as a pair of right $\cA$-modules $F_0,F_1$ together with a map $F_0\otimes_\cA B\to F_1$, that is, maps $\alpha:F_0\to F_1$ and $\beta:F_0\otimes_\cA N\to F_1$.
Here $\alpha$ is the action of $u$, so $F$ is right orthogonal to all $\cofib(u_a)$ if and only if $\alpha$ is a quasi-isomorphism.
Restriction along $\Phi$ sends a $T$-module $G$ to $\Phi^*G=(G,G;\mathrm{id},\beta_G)$, where $\beta_G$ is the action of $N$ on $G$; this lies in the right orthogonal of $\{\cofib(u_a) \} $.
We show that $\Phi^*$ is indeed an equivalence onto the right orthogonal.  Then $\Phi_!$ is a localization with kernel generated by $\{\cofib(u_a) \} $. 

Let $j_i: \cA \to \mathbb P_\cA(N)$, $a \mapsto a_i$, be the two inclusions.
Then $j_{0!}G=(G,G\otimes_\cA B; \mathrm{id})$ and $j_{1!}G=(0,G; 0   )$, and $\Phi_!j_{0!}G=\Phi_!j_{1!}G=G\otimes_\cA T$.
Every $F$ has the resolution
\[
0\to j_{1!}(F_0\otimes_\cA B)\to j_{0!}F_0\oplus j_{1!}F_1\to F\to0.
\]
Applying $\Phi_!$ and cancelling the summand $F_0\otimes_\cA T$, which comes from the summand $\cA_\Delta$ of $B$, gives
\begin{equation} \label{eq: Phi shriek}
\Phi_!F\simeq\cofib\left(F_0\otimes_\cA N\otimes_\cA T\to F_1\otimes_\cA T\right)
\end{equation}
where the map is given by
\[
\qquad f\otimes n\otimes t\mapsto\alpha(f)\otimes nt-\beta(f\otimes n)\otimes t.
\]
For $F=\Phi^*G$ this is the standard resolution
$G\otimes_\cA N\otimes_\cA T\to G\otimes_\cA T\to G$ of a module over a tensor category, so $\Phi_!\Phi^*G\simeq G$, and $\Phi^*$ is fully faithful.

Now let $F$ lie in the right orthogonal of $\{\cofib(u_a) \} $, so $\alpha: F_0 \to F_1$ is a quasi-isomorphism. 
Filter both terms of \eqref{eq: Phi shriek} by the total number of $N$-factors, including the initial $N$ in the source.  The map preserves this filtration.  On the graded piece of degree $m\ge1$ it is
\[
F_0\otimes_\cA N^{\otimes_\cA m}\xrightarrow{\ \alpha\otimes1\ }F_1\otimes_\cA N^{\otimes_\cA m},
\]
which is a quasi-isomorphism.  In degree $0$ the source vanishes and the target is $F_1$.  Hence the inclusion $F_1\to\Phi_!F$ is a quasi-isomorphism of $\cA$-modules.

The two components of the unit $F\to\Phi^*\Phi_!F$ are
\[
F_0\xrightarrow{\alpha}F_1\longrightarrow\Phi_!F,
\qquad
F_1\longrightarrow\Phi_!F.
\]
Both are quasi-isomorphisms, so the unit is an equivalence.
Thus $\Phi^*$ is an equivalence onto the right orthogonal.

\end{proof}

We  need the following lemma, which is the noncommutative analogue of the fact that if $E \to B$ is a bundle and $Z \subset B$ is closed, then $E|_{B \setminus Z} = E \setminus E|_Z$. 

\begin{lemma} \label{lem: tensor localization}
Let $q: \cA \to \cA'$ be a localization with kernel $\mathcal K$.  Then the induced functor $\Perf(T_\cA(N)) \to \Perf(T_{\cA'}(q_!N))$ is a localization with kernel $\thick\{L \otimes_\cA T_\cA(N) : L \in \mathcal K\}$.  In particular, if $\mathcal K=\thick(K)$, the kernel is $\thick(K \otimes_\cA T_\cA(N))$.
\end{lemma}
\begin{proof}
Write $N' = q_!N$,  $T=T_\cA(N)$ and $T'=T_{\cA'}(N')$, let $\Psi:T\to T'$ be the induced dg functor, and let $\iota$ be the fully faithful right adjoint of $q$ on module categories; its image is $\mathcal K^\perp$.
For $L\in\mathcal K$ and a $T$-module $G$ we have $\Hom_T(L\otimes_\cA T,G)\simeq\Hom_\cA(L,G)$.
So $G$ is right orthogonal to all $L\otimes_\cA T$ with $L\in\mathcal K$ if and only if its underlying $\cA$-module lies in $\mathcal K^\perp$, i.e. has the form $\iota G'$.  This holds for $G=\Psi^*G'$.
For such $G$, base change along $q$ gives $(\iota G'\otimes_\cA N)\otimes_\cA\cA'\simeq G'\otimes_{\cA'}N'$, so applying $\Psi_!$ to the standard resolution $G\otimes_\cA N\otimes_\cA T\to G\otimes_\cA T\to G$ yields
\[
\Psi_!G\simeq\cofib\left(G'\otimes_{\cA'}N'\otimes_{\cA'}T'\to G'\otimes_{\cA'}T'\right).
\]
For $G=\Psi^*G'$ this is the standard resolution of $G'$, so $\Psi_!\Psi^*G'\simeq G'$, and $\Psi^*$ is fully faithful.
For general $G$ in the right orthogonal, the filtration argument of Lemma \ref{lem: P1 bundle} shows that $G'\to\Psi_!G$ is a quasi-isomorphism of $\cA'$-modules, so $\Psi_!G=0$ implies $G'=0$, hence $G=0$.
As in Lemma \ref{lem: P1 bundle}, it follows that $\Psi^*$ identifies $\Mod_{T'}$ with the right orthogonal of $\{L\otimes_\cA T:L\in\mathcal K\}$.
\end{proof}

Call a localization $q:\widehat\cC\to\cC$ of dg categories a \emph{categorical compactification} of $\cC$ if $\widehat\cC$ is smooth and proper.

\begin{proposition}\label{prop: compactification-relative-CY}
Let $f:\cA\to\cB$ be a dg functor between smooth dg categories.  Suppose that $q:\widehat\cB\to\cB$ is a categorical compactification and that $f$ lifts to a dg functor $\widehat f:\cA\to\widehat\cB$ with $q\widehat f\simeq f$.  

Consider the $\widehat\cB$-bimodule
$
\widehat J=\cofib\left(
\widehat\cB^!\longrightarrow\widehat f_!(\cA^!)
\right)[n-2]$,
where the map is the dual of the composition map $\widehat f_!(\cA_\Delta)\to\widehat\cB_\Delta$. 
Then there is a categorical compactification:
$
\Perf(\mathbb P_{\widehat\cB}(\widehat J))
\longrightarrow\Perf(\Pi_n(\cB,\cA))
$.

If moreover $\widehat\cB=\thick(G)$ and $\ker q=\thick(K)$ for objects $G$ and $K$, the compactification constructed below has kernel $\thick(Z)$ for a single object $Z$.
\end{proposition}
\begin{proof}
The bimodules $\widehat\cB^!$ and $\widehat f_!(\cA^!)$ are perfect, since $\widehat\cB$ and $\cA$ are smooth and $\widehat f_!$ preserves perfect bimodules; hence so is $\widehat J$.  Thus $\Perf(\mathbb P_{\widehat\cB}(\widehat J))$ is smooth and proper.

Since $q$ is a localization, base change $q_!$ takes $\widehat\cB^!$ to $\cB^!$ and commutes with duals of perfect bimodules \cite[Lemma 2.9 and Proposition 3.10]{keller2009deformed}.  Since $q\widehat f\simeq f$, it takes $\widehat f_!(\cA^!)$ to $f_!(\cA^!)$, and it takes the composition map for $\widehat f$ to the composition map for $f$, hence the dual of the former to $c^!$.  Thus $q_!\widehat J\simeq\cofib(c^!)[n-2]=\Theta_{\cB/\cA}[n-1]$.

By Lemma \ref{lem: P1 bundle}, $\Perf(\mathbb P_{\widehat\cB}(\widehat J))\to\Perf(T_{\widehat\cB}(\widehat J))$ is a localization.
By Lemma \ref{lem: tensor localization}, applied to $q$ and the base change map $\widehat J\to q_!\widehat J\simeq\Theta_{\cB/\cA}[n-1]$, the functor $\Perf(T_{\widehat\cB}(\widehat J))\to\Perf(T_\cB(\Theta_{\cB/\cA}[n-1]))=\Perf(\Pi_n(\cB,\cA))$ is a localization.
Hence the composite
$
\Perf(\mathbb P_{\widehat\cB}(\widehat J))
\longrightarrow\Perf(T_{\widehat\cB}(\widehat J))
\longrightarrow\Perf(\Pi_n(\cB,\cA))
$
is a localization from a smooth and proper dg category, i.e. a categorical compactification.

If $\widehat\cB=\thick(G)$ and $\ker q=\thick(K)$, the kernels of the two functors are $\thick(\cofib(u_G))$ and $\thick(K\otimes_{\widehat\cB}T_{\widehat\cB}(\widehat J))$.
In the notation of the proof of Lemma \ref{lem: P1 bundle}, the generator of the second kernel is the image of $j_{0!}K$, as $\Phi_!j_{0!}K=K\otimes_{\widehat\cB}T_{\widehat\cB}(\widehat J)$.
Hence the kernel of the composite is $\thick(Z)$ with $Z=\cofib(u_G)\oplus j_{0!}K$.
\end{proof}

\begin{proposition} \label{prop: compactification}
    Let $\cB$ be the dg category of Example~\ref{eg:nontrivial ff lcy}. 
    Then there is a finite dimensional dg algebra $\Gamma$ with $\Perf(\Gamma)$ smooth and proper, and an object $Z \in \Perf(\Gamma)$, such that
    \[
    \Perf(\Pi_2(\cB, k)) \simeq \Idem\left( \Perf(\Gamma)/\thick(Z) \right).
    \]
\end{proposition}
\begin{proof}
Preserving the notation of Example \ref{eg:nontrivial ff lcy} and Proposition \ref{prop:cat cpt R}: since $\widetilde R$ is a domain, left multiplication by $\delta-\lambda t$ is injective, so
\[
\overline M_\lambda:=\cofib\left(\cO(-1)\xrightarrow{\ \delta-\lambda t\ }\cO\right)\in\widehat\cR
\]
is represented by the right module $\widetilde R/(\delta-\lambda t)\widetilde R$, and $q_R(\overline M_\lambda)=M_\lambda$.
Hence the composite
\[
q: \widehat\cR \xrightarrow{\ q_R\ } \Perf(R) \longrightarrow \cB
\]
is a localization with kernel $\thick(K_0)$, where $K_0=D_0\oplus D_1\oplus D_2\oplus\overline M_0$, and $\overline X:=\overline M_1$ lifts $X$.

Applying Proposition~\ref{prop: compactification-relative-CY} gives the assertion. In particular, under $\widehat\cR\simeq\Perf(E)$,
with the bimodule
$\widehat J=\cofib(\widehat\cR^!\to\widehat f_!(k))$,
where $\widehat f:k\to\widehat\cR$ selects $\overline M_1$,
we have $\Gamma = \mathbb P_E(\widehat J)$. 
\end{proof}

\appendix

\section{A finite dimensional model}
\label{app:finite-dg-model}

We give a finite dimensional dg algebra $B$, with $\dim_k B = 17$, quasi-isomorphic to the dg algebra $A$ of Example~\ref{dga which does not cogenerate}, together with a two dimensional dg $B$-module $C$ with $\Hom_B(C, B) = 0$.  
\begin{remark}
    Every proper {\em connective} (i.e. $H^{>0}=0$) dg algebra over a field admits a quasi-isomorphic finite dimensional model by a theorem of Raedschelders and Stevenson
    \cite[Appendix~A, Theorem~A.3]{Goodbody2026}.  However, the same is not true for proper \emph{coconnective} (i.e. $H^{<0}=0$) dg algebras, even smooth ones. Indeed, by \cite[Corollary~2.21]{Orlov2020}, the $K_0$ of 
    a finite dimensional smooth and proper dga is a finitely generated abelian group.  But for an elliptic curve $E/\mathbb{C}$, and any classical generator $G$ of $\Perf(E)$,  $\End_E(G)$ is a smooth and proper dg algebra for which
    $$ K_0(\Perf(E)) \cong K_0(E) \cong \mathbb{Z} \oplus \operatorname{Pic}(E).
    $$
    This example is recorded in \cite[Remark~3.1.4]{Goodbody2025Thesis}. See also \cite[Example~3.1.6]{Goodbody2025Thesis} for a counterexample in the module analog of the claim.
\end{remark}

We use the notation and conventions of Section~\ref{sec: example dga}: all modules are right modules, and $\Hom$, $\End$ and $\otimes$ denote derived functors, computed by semi-free resolutions.  Recall that
\[
\Lambda = k\langle x, y \rangle/(x^2, y^2, xy - qyx), \qquad |x| = |y| = 2, \qquad V = (\Lambda/(x+y)\Lambda)[-1] = ku \oplus kv,
\]
with $q \in k^\times$ not a root of unity, $ux = -v$ and $uy = v$.  Paths in a quiver are composed from right to left, and $\langle \ldots \rangle$ denotes the $k$-span.

Consider the quiver
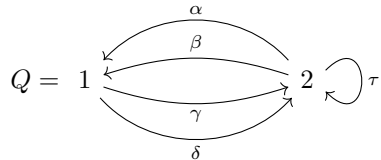
\[
Q = \begin{tikzcd}[column sep=7em]
1 \arrow[r, bend right=18, "\gamma"']
  \arrow[r, bend right=48, "\delta"']
& 2 \arrow[l, bend right=18, "\beta"']
    \arrow[l, bend right=48, "\alpha"']
    \arrow[out=45, in=-45,
       start anchor={[yshift=4pt]east},
       end anchor={[yshift=-4pt]east},
       min distance=9mm, "\tau"]
\end{tikzcd}
\]
with all arrows of degree $1$, and let
\begin{equation}\label{eq:fd17-presentation}
B = kQ/(\beta\delta,\ \delta\beta,\ \beta\tau,\ \tau\delta,\ \tau^2,\ \gamma\beta - q\delta\alpha,\ \gamma\alpha - (q-1)\delta\alpha),
\end{equation}
with the differential determined by
\begin{equation}\label{eq:fd17-differential}
d\beta = \alpha\tau, \qquad d\delta = \tau\gamma, \qquad d\alpha = d\gamma = d\tau = 0,
\end{equation}
and $de_1 = de_2 = 0$ for the vertex idempotents $e_1, e_2$.

\begin{theorem}\label{thm:fd17-model}
The formulas \eqref{eq:fd17-presentation} and \eqref{eq:fd17-differential} define a dg algebra $B$ of total dimension $17$, with
\[
\begin{array}{c|ccccc}
i & 0 & 1 & 2 & 3 & 4 \\ \hline
\dim_k B^i & 2 & 5 & 6 & 3 & 1 \\
\dim_k H^i(B) & 2 & 3 & 3 & 2 & 1 .
\end{array}
\]
There is a zigzag of quasi-isomorphisms of dg algebras between $B$ and $A = \End_\Lambda(\Lambda \oplus V)$, under which $e_1$ corresponds to the projection $e$ onto $\Lambda$.  Moreover, the quotient
\[
C = B/Be_1B = k[\tau]/(\tau^2), \qquad |\tau| = 1, \qquad d_C = 0,
\]
viewed as a right dg $B$-module, satisfies $C \ne 0$ and $\Hom_B(C, B) = 0$.  In particular, $B$ does not cogenerate $\Mod_B$, and this is witnessed by a module which is itself finite dimensional, not merely of finite dimensional cohomology.
\end{theorem}

\subsection{The presentation and its cohomology}

The defining relations imply
\begin{equation}\label{eq:fd17-consequences}
\delta\alpha\delta = \delta\alpha\tau = \tau\gamma\alpha = \tau\gamma\beta = 0.
\end{equation}
Indeed, the first two follow from $\gamma\beta\delta = \gamma\beta\tau = 0$ and $\gamma\beta = q\delta\alpha $; the last two follow by substituting for $\gamma\alpha $ and $\gamma\beta $ and using $\tau\delta = 0$.  These identities show that the differential preserves the ideal of relations; for example
\[
d(\delta\beta) = \tau\gamma\beta - \delta\alpha\tau = 0, \qquad d(\gamma\beta - q\delta\alpha) = -\gamma\alpha\tau - q\tau\gamma\alpha = 0,
\]
where $\gamma\alpha\tau = (q-1)\delta\alpha\tau = 0$, and the remaining relations are checked in the same way.  Also $d^2 = 0$ on the arrows, so $B$ is a dg algebra.

We compute a basis of $B$ by reducing paths using the relations.
Orient the two non-monomial relations as
\[
\gamma\beta \longmapsto q\delta\alpha,
\qquad
\gamma\alpha \longmapsto (q-1)\delta\alpha,
\]
and reduce any path containing one of
\[
\beta\delta,\quad \delta\beta,\quad \beta\tau,\quad
\tau\delta,\quad \tau^2,\quad
\delta\alpha\delta,\quad \delta\alpha\tau
\]
to zero. Each nonzero reduction decreases the number of occurrences
of $\gamma$, so the reduction process terminates.

The only
ambiguities involving a nonzero reduction occur in
$\gamma\beta\delta$ and $\gamma\beta\tau$. Both paths reduce to zero regardless of which rule is applied first.

Induction on the number of occurrences of $\gamma$ shows that
each path has a unique reduced expression. Consequently, the
irreducible paths form a homogeneous $k$-basis of $B$, given
by the following table:
\begin{equation}\label{eq:fd17-basis}
\begin{array}{c|l}
\text{degree} & \text{basis} \\ \hline
0 & e_1, e_2 \\
1 & \alpha, \beta, \gamma, \delta, \tau \\
2 & \alpha\gamma, \beta\gamma, \alpha\delta, \delta\alpha, \alpha\tau, \tau\gamma \\
3 & \alpha\delta\alpha, \alpha\tau\gamma, \delta\alpha\gamma \\
4 & \alpha\delta\alpha\gamma .
\end{array}
\end{equation}
In particular every path of length five is zero.

The nonzero differentials of basis elements are
\[
d\beta = \alpha\tau, \qquad d\delta = \tau\gamma, \qquad d(\beta\gamma) = \alpha\tau\gamma, \qquad d(\alpha\delta) = -\alpha\tau\gamma,
\]
so the cohomology of $B$ has representatives
\begin{equation}\label{eq:fd17-cohomology}
\begin{aligned}
H^0(B) &= \langle e_1, e_2 \rangle, \\
H^1(B) &= \langle \alpha, \gamma, \tau \rangle, \\
H^2(B) &= \langle \alpha\gamma, \beta\gamma + \alpha\delta, \delta\alpha \rangle, \\
H^3(B) &= \langle \alpha\delta\alpha, \delta\alpha\gamma \rangle, \\
H^4(B) &= \langle \alpha\delta\alpha\gamma \rangle .
\end{aligned}
\end{equation}
Here and below a closed element also denotes its cohomology class.

\subsection{Comparison with $A$}

Put
\[
\lambda = -\frac{q(q+1)}{q-1}, \qquad x = \lambda \alpha\gamma + q(\beta\gamma + \alpha\delta), \qquad y = \beta\gamma + \alpha\delta, \qquad h = \alpha\delta\alpha\gamma.
\]
The elements $x, y$ are closed and lie in $e_1Be_1$, and the identities
\[
(\alpha\gamma)^2 = (q-1)h, \qquad (\alpha\gamma)y = qh, \qquad y(\alpha\gamma) = h, \qquad y^2 = 0
\]
give
\begin{equation}\label{eq:fd17-corner}
x^2 = y^2 = 0, \qquad xy = qyx = \lambda q h, \qquad x - qy = \lambda \alpha\gamma.
\end{equation}
We therefore identify $\Lambda$ with the dg subalgebra $\langle e_1, x, y, xy \rangle \subset e_1Be_1$.  The inclusion is a quasi-isomorphism: $e_1Be_1$ has basis $e_1, \alpha\gamma, \beta\gamma, \alpha\delta, \alpha\tau\gamma, \alpha\delta\alpha\gamma $, and its only nonzero differentials are $d(\beta\gamma) = \alpha\tau\gamma $ and $d(\alpha\delta) = -\alpha\tau\gamma $.

Regard $S = Be_1$ as a right dg $\Lambda$-module, so that $e_1S = e_1Be_1 \simeq \Lambda$ and
\[
e_2S = e_2Be_1 = \langle \gamma, \delta, \tau\gamma, \delta\alpha\gamma \rangle, \qquad d\delta = \tau\gamma.
\]
The elements $u = \gamma $ and $v = q\delta\alpha\gamma $ span a dg submodule with zero differential, and
\[
|u| = 1, \quad |v| = 3, \qquad ux = -v, \quad uy = v, \quad vx = vy = 0.
\]
This is $V$, and its inclusion in $e_2S$ is a quasi-isomorphism, the quotient being the contractible complex $\delta \mapsto \tau\gamma $.  Together with the inclusion of $\Lambda$, we obtain a quasi-isomorphism of right dg $\Lambda$-modules
\begin{equation}\label{eq:fd17-object-comparison}
\Lambda \oplus V \to S.
\end{equation}

Next we show that left multiplication identifies $B$ with $\End_\Lambda(S) \simeq \End_\Lambda(\Lambda \oplus V) = A$.  We use the semi-free resolution \eqref{eq: resolution of V} of $V$, whose generators we now call $f_n$, as $e_1, e_2$ are idempotents:
\begin{equation}\label{eq:fd17-resolution}
F = \bigoplus_{n \ge 0} f_n \Lambda, \qquad |f_n| = n+1, \qquad df_0 = 0, \qquad df_{n+1} = f_n\big(x + (-q)^{-n}y\big),
\end{equation}
with $F \to V$ given by $f_0 \mapsto u$ and $f_n \mapsto 0$ for $n \ge 1$.  Recall from Section~\ref{sec: example dga} the resulting computations: $V^* \simeq \Hom_\Lambda(F, \Lambda)$ has cohomology $kw \oplus kz$, in degrees $1$ and $3$, where
\[
w(u) = x - qy, \quad w(v) = xy, \qquad z(u) = xy, \quad z(v) = 0,
\]
and $\End_\Lambda(V) \simeq \Hom_\Lambda(F, V)$ has cohomology $ku_0 \oplus kv_1 \oplus kv_0$, in degrees $0, 1, 2$, where $u_n(f_n) = u$, $v_n(f_n) = v$, and $u_n, v_n$ vanish on the other generators.

Let $\pi_0 : \Lambda \oplus F \to S$ be the composite of $F \to V$ with \eqref{eq:fd17-object-comparison}.  The left action of $B$ on $S$ gives a map of complexes
\begin{equation}\label{eq:fd17-action-comparison}
B \to \Hom_\Lambda(\Lambda \oplus F, S), \qquad c \mapsto L_c \pi_0,
\end{equation}
where $L_c$ denotes left multiplication by $c$.  On cohomology, its four corners are as follows:
\[
\begin{array}{c|c|c}
\text{corner} & \text{representatives in } B & \text{images} \\ \hline
e_1Be_1 & e_1, x, y, xy & 1, x, y, xy \\
e_2Be_1 & \gamma, \delta\alpha\gamma & u, q^{-1}v \\
e_1Be_2 & \alpha, \alpha\delta\alpha & \lambda^{-1}w, (\lambda q)^{-1}z \\
e_2Be_2 & e_2, \tau, \delta\alpha & u_0, (\lambda/q)v_1, q^{-1}v_0 .
\end{array}
\]
The first three rows and the images of $e_2$ and $\delta\alpha $ follow directly from \eqref{eq:fd17-corner}.  For the image of $\tau $, let $\pi_F : F \to e_2S$ be the restriction of $\pi_0$, and define a degree zero map $\eta : F \to e_2S$ by $\eta(f_0) = \delta $ and $\eta(f_n) = 0$ for $n \ge 1$.  Since $\delta x = \lambda \delta\alpha\gamma $ and $\delta y = 0$,
\[
(d\eta)(f_0) = \tau\gamma, \qquad (d\eta)(f_1) = -\lambda \delta\alpha\gamma,
\]
and all other values vanish.  Thus $L_\tau\pi_F - d\eta$ factors through $V \subset e_2S$ and sends $f_1$ to $\lambda \delta\alpha\gamma = (\lambda/q)v$, which gives the last entry of the table.  All displayed scalars are nonzero, so \eqref{eq:fd17-action-comparison} is a quasi-isomorphism.
  
The left $B$-action on $S$ induces a morphism
$B\to\End_\Lambda(S)$ in the homotopy category of dg algebras,
represented on underlying complexes by
\eqref{eq:fd17-action-comparison}.
The preceding computation shows that this morphism is a
quasi-isomorphism. Together with
\eqref{eq:fd17-object-comparison}, this gives
\[
B\simeq\End_\Lambda(S)
 \simeq\End_\Lambda(\Lambda\oplus V)=A.
\]
Under this identification, $e_1$ corresponds to the projection
onto $\Lambda$, since $e_1S\simeq\Lambda$ and $e_2S\simeq V$.

\subsection{The two dimensional witness}

Under the identification of $B$ with $A$, the idempotent $e_1$ corresponds to $e$, so the module $M$ of Lemma~\ref{lem: general noncogeneration} becomes
\[
M = \cofib\left( Be_1 \otimes_{e_1Be_1} e_1B \to B \right) \in \Mod_B.
\]
By Lemma~\ref{lem: general noncogeneration}, $\Hom_B(M, B) = 0$ and $M \simeq \cofib(\mathrm{ev} : V \otimes_\Lambda V^* \to \End_\Lambda(V))$, the latter living in the $e_2$-$e_2$ corner.  By \eqref{eq: V tensor V dual} and the computation following it, $H^*(V \otimes_\Lambda V^*) = k[-2]$ and $\mathrm{ev}$ is injective on it, with image spanned by $v_0$.  Hence $H^*(M) = k \oplus k[-1]$, with representatives $u_0$ and $v_1$, that is, by the table above, the images of $e_2$ and $\tau $.

The canonical map
\[
M \to C = B/Be_1B = k[\tau]/(\tau^2)
\]
sends these representatives to the basis $e_2, \tau $ of $C$ and is therefore a quasi-isomorphism.  Consequently $\Hom_B(C, B) = \Hom_B(M, B) = 0$, which proves the theorem.

Explicitly, $C$ has basis $m_0 = e_2$, $m_1 = \tau $ in degrees $0, 1$, zero differential, and right $B$-action
\[
m_i e_2 = m_i, \qquad m_i e_1 = 0, \qquad m_0 \tau = m_1, \qquad m_1 \tau = 0,
\]
with $\alpha, \beta, \gamma, \delta $ acting by zero.

\section{Cogeneration and the chord conjecture} \label{sec: chord}

Arnol'd's chord conjecture \cite{Arnold}  asserts that any compact Legendrian in any compact contact manifold has a Reeb chord. 
This conjecture is known for boundaries of Weinstein domains which are subcritical \cite{Mohnke} or more generally have vanishing symplectic cohomology \cite{Zhou-minimal}; it is also known for all three dimensional contact manifolds \cite{Hutchings-Taubes}.  There has been a spate of recent  progress \cite{Brocic-Cant-Shelukhin, Guo-Zhou, Brocic-Cant, shelukhin-orderability}, but the conjecture remains wide open in general.  

Some recent developments in categorical symplectic topology 
can be combined to the following observation: 


\begin{theorem} \label{thm: no reeb gives ff lcy}
    Let $W$ be a Weinstein manifold, and $L \subset \partial_\infty W$ a Legendrian.  If $L$ admits no Reeb chords, then the \cite{GPS1}  functor $\iota: \Fuk(T^*L) \to \Fuk(W, L)$ is a fully faithful left Calabi-Yau morphism with nonzero domain. 
\end{theorem}
\begin{proof}[Sketch of proof]
It follows from a geometric construction \cite{BEE, Ekholm-Lekili} that if $L$ has no Reeb chords, then $\iota$ is fully faithful.  The category $\Fuk(T^*L)$ is nonzero. 
On the other hand, it follows from \cite{shende-takeda} translated through \cite{GPS3} that the morphism $\iota$ is left Calabi-Yau.\footnote{Strictly speaking, the use of \cite{GPS3} means we should assume $W$ is polarizable.  Experts know how to remove this hypothesis (e.g. one runs \cite[Sec. 11]{nadler-shende} in the Fukaya category); or alternatively, it is expected that one can run a `relative' version of \cite{ganatra-thesis, ganatra-cyclic} or \cite{legout} to construct the Calabi-Yau structure directly, without passage through the sheaf/Fukaya dictionary.} 
\end{proof}

\begin{remark}
To show that Theorem \ref{thm: no reeb gives ff lcy} has content, we rederive Zhou's generalization \cite{Zhou-minimal} of Mohnke's theorem \cite{Mohnke}: the contact boundary of a Weinstein domain with vanishing symplectic homology satisfies the Chord conjecture.  
Indeed, if $SH(W) = HH^*(\Fuk(W)) = 0$, then $\Fuk(W) = 0$.  For a Legendrian $L$, by \cite{GPS2}, one has $\Fuk(W) \cong \Fuk(W, L) / \iota(\Fuk(T^*L))$.  If $L$ has no chords, then by Theorem \ref{thm: no reeb gives ff lcy}, $\iota$ is fully faithful, so by Proposition \ref{universal ff lcy}, the inverse  dualizing bimodule of $\Fuk(W, L)$ vanishes on the split-generating subcategory $\iota(\Fuk(T^*L))$, hence vanishes.  Taking bimodule dual, the diagonal bimodule must also vanish, and hence $\Fuk(W, L) = 0$.  But $\Fuk(W, L)$ contains the nonzero subcategory $\Fuk(T^*L)$; contradiction.
\end{remark}

Let $W$ be a Weinstein manifold and $L \subset \partial_\infty W$ a compact Legendrian.  Suppose that $L$ can be enlarged to a (singular) `total stop'\footnote{We expect this to follow from a relative version of the Giroux-Pardon construction, giving a Lefschetz fibration for $W$ such that $L$ is a component of the core of the fiber. In any case, we can also obtain such a localization sequence after passing to microsheaves, as described below.} $\Lambda \supset L$, so that $\Fuk(W, \Lambda)$ is smooth and proper.  Then per \cite{GPS2}, we have a localization sequence
\begin{equation} \label{eq: quotient by linking disks}
0 \to \langle D_i \rangle \to \Fuk(W, \Lambda) \to \Fuk(W, L) \to 0
\end{equation}
where the $D_i$ are the `linking disks' to $\Lambda \setminus L$; finitely many of them suffice.  Under the dictionary of \cite{GPS3}, the linking disks are the corepresentatives of the microstalk functors.  Since $\Fuk(W, \Lambda)$ is proper,
\[
A_{\Lambda, L} := \End\Big( \bigoplus_i D_i \Big)
\]
is a proper dg algebra.

\begin{corollary} \label{prop: cogeneration implies chord}
    If $A_{\Lambda, L}$ cogenerates, then $L$ has a Reeb chord.
\end{corollary}
\begin{proof}
    Follows from Proposition \ref{cogeneration versus nondegeneracy}, Proposition \ref{universal ff lcy}, and Theorem \ref{thm: no reeb gives ff lcy}.
\end{proof}

We now explain why we were specifically interested in {\em coconnective} dg algebras.  

A sequence as in \eqref{eq: quotient by linking disks} can be obtained from microlocal sheaf theory, as follows.  Let $(W, L)$ be a stopped Weinstein manifold, with relative skeleton $\mathfrak c_{W, L}$.  By \cite[Theorem 1.4]{GPS3}, $\Fuk(W, L)$ is equivalent, up to passing to the opposite category, to the category $\mu sh(\mathfrak c_{W, L})^c$ of compact objects in microlocal sheaves on $\mathfrak c_{W, L}$.  Embed $W$ as a Liouville hypersurface in a cosphere bundle $S^*M$.  The doubling trick \cite[Theorem 7.17]{GPS3} (see \cite[Theorem 7.30]{nadler-shende}) identifies $\mu sh(\mathfrak c_{W, L})$ with the category of sheaves on $M \times \bR$ which are microsupported in the double of $\mathfrak c_{W, L} \times (0, 1)$ and have vanishing stalk at infinity.  Now let $X$ be a closed manifold containing the support of these sheaves, and $\mathcal T$ a Whitney triangulation of $X$ such that the union $T^*_{\mathcal T} X$ of the conormals to the strata (note that it contains the zero section) contains the double.  Then $\Sh(X, T^*_{\mathcal T} X)^c$, which is equivalent to perfect modules over the poset $\mathcal T$, is smooth and proper \cite[Proposition 4.25]{GPS3}, and $\mu sh(\mathfrak c_{W, L})^c$ is its quotient by corepresentatives of microstalks at smooth points of $T^*_{\mathcal T} X$ \cite[Corollary 4.23]{GPS3}:
\[
0 \to \langle \chi_{\xi_i} \rangle \to \Sh(X, T^*_{\mathcal T} X)^c \to \mu sh(\mathfrak c_{W, L})^c \to 0,
\]
where $\chi_\xi$ denotes the corepresentative of the microstalk functor $\mu_\xi$.

\begin{definition}
    We say a Whitney triangulation $\mathcal T$ is \emph{good}, if it satisfies
    \begin{enumerate}
        \item For each $x \in \sigma \in \mathcal T$, identify the neighborhood of $x$ with a small open ball in $T_x X$ and take a subspace $V_{\sigma, x}$ complementary to $T_x \sigma$.  Then $\mathcal T$ restricted to $V_{\sigma, x}$ is locally analytically diffeomorphic to a fan $\Sigma_{\sigma, x}$.
        \item Each cone in $\Sigma_{\sigma, x}$ is convex.
    \end{enumerate}
\end{definition}

\begin{remark}
    While we do not necessarily expect that any $L$ can be isotoped to be a subset of some $T^*_\mathcal{T}X$ for good $\mathcal{T}$, we do expect that a variant of Nadler's noncharacteristic expansion \cite[Sections 5--6]{Nadler-expansion} can be used to show that $\mu sh(\mathfrak c_{W, L})$ can always be realized as a quotient of some $\Sh(X, T^*_{\mathcal T} X)^c$ for good $\mathcal T$.
\end{remark}



Let us recall some notation from \cite{KS90}.  Let $S_1, S_2$ be two subsets of $X$, and $x \in X$.  Then $C_x(S_1, S_2)$ is the conic subset of $T_x X$ defined by: $v \in C_x(S_1, S_2)$ if and only if there exists a sequence $\{ (x_n, y_n, c_n) \}$ in $S_1 \times S_2 \times \bR^+$ such that
\[
x_n \rightarrow x, \quad y_n \rightarrow x, \quad c_n (x_n - y_n) \rightarrow v.
\]
Denote $C_x(S) := C_x(S, x)$.  For $x \in X$ and a subset $S \subset X$, define
\[
N_x(S) = T_x X \setminus C_x(X \setminus S, S), \qquad N^*_x(S) = N_x(S)^\circ.
\]
Here for a conic set $\gamma$, we denote $\gamma^\circ = \{ v \mid \forall u \in \gamma, \, \langle v, u \rangle \ge 0 \}$.  Note that for a convex cone $\gamma$ in a vector space $V$, we have $N_0(\gamma) = C_0(\gamma)$.

\begin{proposition}[{\cite[Proposition 5.3.8]{KS90}}] \label{prop:SS_of_open_closed}
    For an open subset $U \subset X$ and a closed subset $Z \subset X$, we have
    \[
    \SS(k_U) \subset \bigcup_{x \in \overline U} -N_x^*(U), \qquad \SS(k_Z) \subset \bigcup_{x \in Z} N_x^*(Z).
    \]
\end{proposition}

We consider generic points $\xi = (x, p) \in T^*_{\mathcal T} X$.  Here \emph{generic} means that if $x \in \sigma \in \mathcal T$, the subspaces spanned by the cones in $\Sigma_{\sigma, x}$ and $\xi^\perp \subset V_{\sigma, x}$ are transverse to each other.  Locally, take a small open ball $B_\epsilon(x)$ around $x$, and identify it with a small ball in $T_x X$.  Then define
\[
B_\xi^+ := \{ v \in B_\epsilon(x) \mid \langle \xi, v \rangle \ge -\epsilon^2 \}.
\]
A cartoon picture for $B_\xi^+$ and $\starp(\xi)$, defined below, is given in Figure~\ref{fig:cartoon_B_starp}.
    Denote the microlocal stalk functor at $\xi$ by $\mu_\xi$.  Then by the definition, for $\cF \in \Sh(X, T^*_{\mathcal T} X)$, we have
    \[
    \mu_\xi(\cF) = \Hom(k_{B_\xi^+}, \cF).
    \]

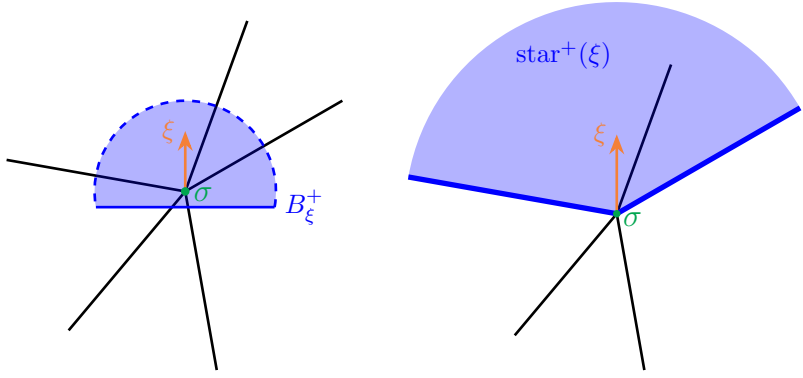
\begin{figure}
\centering
\begin{tikzpicture}[scale = .8, line width = 1pt]
    \draw[line width = 1pt] (0, 0) -- ++(30: 3);
    \draw[line width = 1pt] (0, 0) -- ++(70: 3);
    \draw[line width = 1pt] (0, 0) -- ++(170: 3);
    \draw (0, 0) -- ++(230:3);
    \draw (0, 0) -- ++ (-80:3);

    \draw[blue] (-10:1.5) -- (190:1.5);
    \draw[blue, dashed] (-10:1.5)
    arc[start angle= -10, end angle=190, radius=1.5];

    \fill[blue, opacity = .3] (-10:1.5)
    arc[start angle=-10, end angle=190, radius=1.5]
    -- cycle;

    \draw[-Stealth, Orange, line width = 1pt] (0,0) -- (0, 1);
    \node[left] at (0, 1) {\textcolor{Orange}{$\xi$}};

    \fill[Green] (0,0) circle (2pt);
    \node at (.3, -.05) {\Large \textcolor{Green}{$\sigma$}};

    \node[right] at (-10:1.5) { \textcolor{blue}{$B_\xi^+$}};
\end{tikzpicture}
\qquad
\begin{tikzpicture}[scale = .7, line width = 1pt]
    \draw[line width = 1pt] (0, 0) -- ++(30: 3);
    \draw[line width = 1pt] (0, 0) -- ++(70: 3);
    \draw[line width = 1pt] (0, 0) -- ++(170: 3);
    \draw (0, 0) -- ++(230:3);
    \draw (0, 0) -- ++ (-80:3);

    \draw[blue, line width = 2pt] (30:4) -- (0,0) -- (170:4);

    \fill[blue, opacity = .3] (30:4)
    arc[start angle=30, end angle=170, radius=4] -- (0, 0) -- cycle;

    \draw[-Stealth, Orange, line width = 1pt] (0,0) -- (0, 1.5);
    \node[anchor = east] at (0, 1.5) {\textcolor{Orange}{$\xi$}};

    \fill[Green] (0,0) circle (2pt);
    \node at (.3, -.1) {\Large \textcolor{Green}{$\sigma$}};

    \node at (-1, 3) { \textcolor{blue}{$\starp(\xi)$}};
\end{tikzpicture}
\caption{A cartoon picture for $B_\xi^+$ and $\starp(\xi)$}
\label{fig:cartoon_B_starp}
\end{figure}

\begin{definition}
    For a point $\xi = (x, p) \in T^*_{\mathcal T} X$, define
    \[
    \starp(\xi) = \bigcup_{\substack{\sigma \in \mathcal T, \, x \in \overline{\sigma} \\ \langle \xi, C_x(\sigma) \rangle \ge 0}} \sigma.
    \]
\end{definition}

For $x \in \sigma \in \mathcal T$, each stratum $\tau$ in $\sta(\sigma)$ gives rise to a cone $\gamma_\tau$ in the fan $\Sigma_{\sigma, x}$ in $V_{\sigma, x}$.  Then $\starp(\xi)$ is the union of the strata $\tau \subset \sta(\sigma)$ satisfying $\langle \gamma_\tau, \xi \rangle \ge 0$.

\begin{proposition} \label{prop:corep_microstalk}
    Let $\mathcal T$ be a good triangulation and $\xi$ generic.  Then there is an equivalence of functors from $\Sh(X, T^*_{\mathcal T} X)^c$ to $\Perf(k)$:
    \[
    \Hom(k_{\starp(\xi)}, -) \cong \Hom(k_{B^+_\xi}, -).
    \]
    In other words, the microlocal stalk functor at $\xi$ is corepresented by $\chi_\xi = k_{\starp(\xi)}$.
\end{proposition}
\begin{proof}[Sketch of proof]
    Take a neighborhood of $x$, locally isomorphic to $\sigma \times V_{\sigma, x}$.  By our assumption on $\mathcal T$ and $\xi$, locally $\mathcal T$ is analytically isomorphic to $\sigma \times \Sigma_{\sigma, x}$.  Define $C_\xi^+ = B_\epsilon(x) \cap \starp(\xi) \subset \sigma \times V_{\sigma, x}$.  Then topologically, $C_\xi^+$ is isomorphic to $\starp(\xi)$, hence there is a non-characteristic deformation between $k_{C_\xi^+}$ and $k_{\starp(\xi)}$.

    On the other hand, take $\underline B_\xi^+ := \{ v \in B_\epsilon(x) \mid \langle \xi, v \rangle \ge 0 \}$.  There is a non-characteristic deformation between $k_{B_\xi^+}$ and $k_{\underline B_\xi^+}$.  Hence, we only need to construct a non-characteristic deformation between $k_{\underline B_\xi^+}$ and $k_{C_\xi^+}$.

    For a point $v \in \underline B^+_\xi$, assume $v \in \rho \in \Sigma_{\sigma, x}$.  By our assumption, $\rho$ is a cone over a polyhedron $P_\rho$ on $\partial B_\epsilon(x)$.  Since $\rho$ is convex, $P_\rho$ must have at least one vertex in $\underline B^+_\xi$.  Let $e_1, e_2, \dots, f_1, f_2, \dots$ be the vertices of $P_\rho$, where the $e_i$ are the vertices in $\underline B^+_\xi$.  Write $v = \sum_i c_i e_i + \sum_j d_j f_j$, and define
    \[
    H_t(v) = \Big( (1-t) \sum_i c_i e_i + \sum_j d_j f_j \Big) \frac{\| v \|}{\big\| (1-t) \sum_i c_i e_i + \sum_j d_j f_j \big\|}.
    \]
    Then $H_t$ is a map from $\underline B_\xi^+$ to $C_\xi^+$.  Taking $\cF_t = k_{H_t(\underline B_\xi^+)}$, we get a deformation from $k_{\underline B_\xi^+}$ to $k_{C_\xi^+}$. A small negative push-off of this deformation would provide a non-characteristic deformation.  
\end{proof}

In particular, $\chi_i$ is represented by an ordinary sheaf in degree $0$, hence $\End (\bigoplus_i \chi_i) $ is coconnective. 

\begin{lemma} \label{lem:H0_of_linking_disks}
    Let $\xi_1 = (x_1, p_1)$ and $\xi_2 = (x_2, p_2)$ be generic points on the smooth strata of $T^*_{\mathcal T} X$.  Then $H^0 \Hom(\chi_{\xi_1}, \chi_{\xi_2})$ is either $k$ or $0$.  Moreover, if $H^0 \Hom(\chi_{\xi_1}, \chi_{\xi_2}) = k$, then
    \[
    x_1 \in \starp(\xi_2) \quad \text{and} \quad \langle \xi_1, N_{x_1}(\starp(\xi_2)) \rangle \ge 0,
    \]
    and $\starp(\xi_2)$ is convex at $x_1$, hence $N_{x_1}(\starp(\xi_2)) = C_{x_1}(\starp(\xi_2))$.
\end{lemma}
\begin{proof}
    Let $B'_{\xi_1} = \{ v \in \overline{B_\epsilon(x_1)} \mid \langle \xi_1, v \rangle > -\epsilon^2 \}$.  Then we have
    \begin{align*}
        \Hom(k_{\starp(\xi_1)}, k_{\starp(\xi_2)}) &= \Hom(k_{B_{\xi_1}^+}, k_{\starp(\xi_2)}) \\
        &= \Gamma\big( k_{B_{\xi_1}'} \otimes k_{\starp(\xi_2)} \big) \\
        &= \Gamma_c\big( B'_{\xi_1} \cap \starp(\xi_2) \big).
    \end{align*}
    Assume $H^0 \Hom(k_{\starp(\xi_1)}, k_{\starp(\xi_2)}) = H^0_c(B'_{\xi_1} \cap \starp(\xi_2)) \neq 0$.  Then $B'_{\xi_1} \cap \starp(\xi_2)$ has no open boundary.  This implies that $x_1$ must lie on the closed boundary of $\starp(\xi_2)$.  Hence $x_1 \in \starp(\xi_2)$, and by Proposition~\ref{prop:SS_of_open_closed}, we have $\langle \xi_1, N_{x_1}(\starp(\xi_2)) \rangle \ge 0$.  Moreover, the open boundary of $B'_{\xi_1}$ cannot intersect the interior of $\starp(\xi_2)$.  This forces $\starp(\xi_2)$ to be convex at $x_1$.
\end{proof}

\begin{lemma}
    Let $\xi_1$ and $\xi_2$ be as above, with $x_1 \in \sigma_1$ and $x_2 \in \sigma_2$.  If $H^0 \Hom(\chi_{\xi_1}, \chi_{\xi_2}) = k$, then either
    \[
    \dim(\sigma_1) > \dim(\sigma_2),
    \]
    or
    \[
    \sigma_1 = \sigma_2 \quad \text{and} \quad \starp(\xi_1) \supset \starp(\xi_2).
    \]
\end{lemma}
\begin{proof}
    Assume $H^0 \Hom(\chi_{\xi_1}, \chi_{\xi_2}) = k$.  By Lemma~\ref{lem:H0_of_linking_disks}, $x_1 \in \starp(\xi_2)$, which implies that $\overline \sigma_1 \supset \sigma_2$.  Therefore, either $\dim(\sigma_1) > \dim(\sigma_2)$ or $\sigma_1 = \sigma_2$.  If $\sigma_1 = \sigma_2$, then again by Lemma~\ref{lem:H0_of_linking_disks}, $\big\langle \xi_1, C_{x_1}\big( \starp(\xi_2) \big) \big\rangle \ge 0$.  By the definition of $\starp$, we have $\starp(\xi_2) \subset \starp(\xi_1)$.
\end{proof}

Therefore, we can define an order on the microstalks associated to the smooth strata of $T^*_{\mathcal T} X$, such that
\[
H^0 \Hom(\chi_{\xi_1}, \chi_{\xi_2}) = 0 \quad \text{if} \quad \xi_1 > \xi_2.
\]

\begin{lemma}
    Consider a finite collection of microstalks $\{ \chi_i \}$ at smooth points on $T^*_{\mathcal{T}} X$ for a good triangulation $\mathcal{T}$.  Then $H^0 \End\left( \bigoplus_i \chi_i \right)$ is of finite global dimension.
\end{lemma}
\begin{proof}
    This is the path algebra of an acyclic quiver with admissible relations.  By \cite[Proposition 3.1.7]{Derksen-Weyman}, it is of finite global dimension.
\end{proof}

Thus the algebras $A = \End\left( \bigoplus_i \chi_i \right)$ for the microsheaf analogue of Proposition \ref{prop: cogeneration implies chord} are proper and coconnective, and $H^0(A)$ is directed, in particular of finite global dimension, and hence cogenerates by~\cite{Rickard}.  

However, Example \ref{dga which does not cogenerate} gives a dg algebra sharing all these properties, which does not cogenerate.

\begin{question}
    Let $\{ \chi_i \}$ be a finite collection of corepresentatives of microstalks as above.  Does the dg algebra $\End\left( \bigoplus_i \chi_i \right)$ cogenerates its module category?
\end{question}


\bibliographystyle{plain}
\bibliography{refs}

\end{document}